\documentclass[a4paper,11pt]{amsart}

\usepackage{amsmath,amssymb,amsthm,mathtools}
\usepackage{geometry}
\usepackage{hyperref}
\hypersetup{colorlinks=true,linkcolor=blue,citecolor=blue,urlcolor=blue}

\allowdisplaybreaks

\numberwithin{equation}{section}

\theoremstyle{plain}
\newtheorem{thm}{Theorem}[section]
\newtheorem{prop}[thm]{Proposition}
\newtheorem{lem}[thm]{Lemma}
\newtheorem{coro}[thm]{Corollary}
\newtheorem{mainthm}{Theorem}
\newtheorem{defn}[thm]{Definition}
\newtheorem{rem}[thm]{Remark}

\newcommand{\Cat}{\operatorname{Cat}}
\newcommand{\dif}{\mathrm{d}}

\title[Compact Helicoidal Minimal Surfaces]
{Compactness and Willmore Energy of Helicoidal Minimal Surfaces
in the 3-Sphere}

\author[J. Q. Ge]{Jianquan Ge}
\address{School of Mathematical Sciences, Laboratory of Mathematics and Complex Systems, Beijing Normal University, Beijing 100875, P. R. China.}
\email{jqge@bnu.edu.cn}

\author[S. L. Li]{Shilin Li$^{*}$}
\address{School of Mathematical Sciences, Laboratory of Mathematics and Complex Systems, Beijing Normal University, Beijing 100875, P. R. China.}
\email{lishilin@mail.bnu.edu.cn}

\date{}
\thanks{$^{*}$ Corresponding author.}
\thanks{J. Q. Ge is partially supported by NSFC (No. 12571049) and the Fundamental Research Funds for the Central Universities.}
\begin{document}

\begin{abstract}
Recently, I.~Castro, I.~Castro-Infantes, and J.~Castro-Infantes
introduced a two-parameter family of helicoidal minimal surfaces in
$\mathbb S^3$, denoted by $\operatorname{Hel}_c^h$, with the pitch
$h\geq0$ and $c\in[0,1/2)$. At
$(h,c)=(0,0)$, the surface $\operatorname{Hel}_0^0$ is the totally
geodesic sphere, while the limiting surface as $c\to1/2^-$ is the
Clifford torus. The subfamily $c=0$, $h>0$, consists of the Lawson
spherical helicoids, whereas the subfamily $h=0$, $0<c<1/2$,
consists of the spherical catenoids, whose compact members are the Otsuki
tori. Castro et al.\ remarked that it is not an easy problem to determine when $\operatorname{Hel}_c^h$ is a compact surface.

In this paper, we resolve this compactness problem, namely we prove that
the compact members of the family are characterized by
\[
\operatorname{Hel}_c^h \text{ is compact}
\quad\Longleftrightarrow\quad
\begin{cases}
h\in\mathbb Q, & c=0,\\[1mm]
h\in\mathbb Q\ \text{and}\ q(h,c)\in\mathbb Q,
& 0<c<1/2,
\end{cases}
\]
where $q(h,c)$ is given by an explicit integral. For $0<c<1/2$, 
every compact quotient surface induced by the
parametrization is a torus. For $c=0$ and $h=j/\nu>0$ written in
lowest terms, the quotient of the parameter plane by the full automorphism
group is a torus when
$j$ and $\nu$ are both odd and a Klein bottle otherwise. 
The Willmore energies of the corresponding compact immersed surfaces are
computed explicitly. Along each Lawson associated family of a spherical
catenoid, only finitely many parameter values yield compact helicoidal
surfaces with Willmore energy below any prescribed bound.
\end{abstract}

\maketitle

\section{Introduction}

Minimal surfaces in the round three-sphere form a classical setting in
which local differential geometry, symmetry, and global topology interact
in an especially explicit way. The simplest closed examples are the
totally geodesic sphere and the Clifford torus. Lawson constructed broad
families of compact minimal surfaces in $\mathbb S^3$
\cite{Lawson1970}, while Otsuki discovered a countable family of
rotationally invariant immersed minimal tori \cite{Otsuki1970}. The
symmetry-reduction approach to such examples was developed more generally
by Hsiang and Lawson \cite{HsiangLawson1971}. Rotational minimal surfaces
in space forms were subsequently studied by do Carmo and Dajczer
\cite{doCarmoDajczer1983}, and those in $\mathbb S^3$ were studied
further by Ripoll \cite{Ripoll1989}. For a broader overview of minimal
surfaces in $\mathbb S^3$, see the excellent survey by Brendle \cite{Brendle2013Survey}. These
constructions illustrate a recurring phenomenon: local differential
equations may produce continuous families of minimal immersions, whereas
compactness requires additional periodicity conditions.

Castro, Castro-Infantes, and Castro-Infantes introduced a two-parameter
family of helicoidal minimal surfaces in $\mathbb S^3$, denoted by
$\operatorname{Hel}_c^h$, with $h\geq0$ and $0\leq c<1/2$, where
$h$ is the pitch \cite{Castro2024}.  At
$(h,c)=(0,0)$, the surface $\operatorname{Hel}_0^0$ is the totally
geodesic sphere, while the limiting surface as $c\to1/2^-$ is the
Clifford torus. The subfamily $c=0$, $h>0$, consists of the Lawson spherical
helicoids, whereas the subfamily $h=0$, $0<c<1/2$, consists of the
spherical catenoids. The compact spherical catenoids are precisely the Otsuki tori
\cite{Otsuki1970,Castro2024}, whose Willmore energies were studied by Hu
and Song \cite{HuSong2012}. For $h>0$ and $0<c<1/2$, Castro et al.\
obtained local helicoidal parametrizations and identified the resulting
surfaces within the associated families of spherical catenoids. Their
local description, however, does not determine when these parametrizations
close globally; this problem was left open in
\cite[Remark~5.9]{Castro2024}.

The present paper determines exactly which of these helicoidal surfaces
are compact. More precisely, the local profile
curves are first extended globally, yielding globally defined helicoidal
minimal immersions. Necessary and sufficient rationality conditions for compactness are then
established, the automorphism groups are determined, and compact
realizations are described in terms of finite-index subgroups. Explicit
formulas are obtained for the Willmore energies of the quotient surfaces
by the automorphism groups. These results are further applied to the
Lawson associated families of spherical catenoids. Using the identification
of their members with helicoidal minimal surfaces in
\cite[Theorem~5.6]{Castro2024}, we prove that, for each fixed spherical
catenoid, only finitely many members are compact and have
quotients by their automorphism groups with Willmore energy below a prescribed
bound.

For $0<c<1/2$, the local unit-speed profile curve is extended smoothly
to a globally defined curve $\xi_c^h$, whose height function $z(s)$ is
smooth and periodic. This produces a
global helicoidal minimal immersion
\[
X_c^h=X^h(\xi_c^h):
\mathbb R^2\longrightarrow\mathbb S^3.
\]
This construction is carried out in
Section~\ref{sec:global_immersions}, and the resulting immersion is shown
to be minimal in Proposition~\ref{prop:global_minimal_immersion}. For $c=0$, the
corresponding global parametrization is
\[
X_0^h(s,t)
=
\bigl(\cos s\,e^{iht},\sin s\,e^{it}\bigr).
\]
For $0<c<1/2$, the longitude increment computed in
Lemma~\ref{lem:longitude_increment} is normalized to define $q(h,c)$
in Definition~\ref{def:q}. Here $\operatorname{Hel}_c^h$ is called compact if it has a compact
immersed realization locally parametrized by $X_c^h$; the precise
definition is given in Definition~\ref{def:compactness}.

\begin{mainthm}[Compactness criterion]\label{mainthm:compactness}
For $h\geq0$ and $0\leq c<1/2$,
\[
\operatorname{Hel}_c^h \text{ is compact}
\quad\Longleftrightarrow\quad
\begin{cases}
h\in\mathbb Q,
& c=0,\\[1mm]
h\in\mathbb Q\ \text{and}\ q(h,c)\in\mathbb Q,
& 0<c<1/2.
\end{cases}
\]
\end{mainthm}

Thus compactness is characterized by the rationality conditions above.
The condition $h\in\mathbb Q$ gives a period in the $t$-direction,
while $q(h,c)\in\mathbb Q$ gives a period in the $s$-direction
after one full period of the height function. The cases $0<c<1/2$ and
$c=0$ are treated in Theorems~\ref{thm:positive_c_compactness}
and~\ref{thm:lawson_compactness}, and combined in
Theorem~\ref{thm:compactness}. In the rotational subfamily $h=0$, the
criterion recovers the classical closing condition for spherical
catenaries; the resulting compact rotational surfaces are the Otsuki
tori. The range result for $q(h,c)$ also implies that infinitely many
members with $0<c<1/2$ are compact; see
Corollary~\ref{cor:infinitely_many_compact}.

For a compact member with $0<c<1/2$, let $Q_c^h$ denote the quotient
of $\mathbb R^2$ by the automorphism group of $X_c^h$. Then $Q_c^h$ is
a torus, and every compact immersion locally modeled on $X_c^h$ is
induced by a finite-sheeted covering of $Q_c^h$. When $c=0$ and
$h=j/\nu>0$ is written in lowest terms, let $Q_0^h$ denote the quotient
surface by the automorphism group of $X_0^h$. It is a torus when $j$
and $\nu$ are both odd, and a Klein bottle when exactly one of $j$ and
$\nu$ is even. This distinction arises from the presence of glide
reflections in the automorphism group. The corresponding descriptions
in terms of finite-index subgroups are given in
Theorems~\ref{thm:positive_c_compact_realization_classification}
and~\ref{thm:lawson_compact_realization_classification}.

\begin{mainthm}[Willmore energy]\label{mainthm:willmore}
Suppose that $(h,c)\neq(0,0)$ and that
$\operatorname{Hel}_c^h$ is compact. The Willmore energy of $Q_{c}^{h}$ is an explicit positive integer
multiple of $\mathcal E(h,c)$, where $\mathcal E(h,c)$ is defined by
integrating the induced area measure over a specified parameter rectangle.
For $0<c<1/2$, this quantity is evaluated in terms of complete elliptic
integrals.
\end{mainthm}

The explicit evaluation of $\mathcal E(h,c)$ is given in
Theorem~\ref{thm:parameter_rectangle_integral}. The energies of the quotient surfaces by the automorphism groups are
computed in Theorems~\ref{thm:total_energy_lattice_index}
and~\ref{thm:lawson_quotient_energy}.

Our final main result, proved in
Theorem~\ref{thm:finiteness}, is the following.

\begin{mainthm}[Finiteness under an energy bound]
\label{mainthm:finiteness}
Fix a spherical catenoid and $W_0>0$. Along its Lawson associated
family, only finitely many parameter values yield compact helicoidal
surfaces whose quotients by the automorphism groups have Willmore
energy at most $W_0$.
\end{mainthm}

The paper is organized as follows. Section~\ref{sec:global_immersions}
recalls the local reconstruction equations, constructs the global helicoidal
minimal immersions $X_c^h$, and computes the longitude increment of their
profile curves. Section~\ref{sec:compactness} introduces the normalized
increment $q(h,c)$, determines the local symmetries of the parametrizations,
and uses a developing-map argument to prove the compactness criterion and study
the range of $q(h,c)$.
Section~\ref{sec:compact_quotients} determines the automorphism groups
and describes the compact quotient surfaces arising from the parametrizations
in terms of finite-index subgroups. In Section~\ref{sec:willmore_energy}, the
integral of the induced area measure over a parameter rectangle is evaluated,
formulas for the Willmore energies of the quotient surfaces by the
automorphism groups are derived, and the parameter dependence needed for the subsequent
finiteness argument is analyzed. Finally,
Section~\ref{sec:associated_families} applies these results to Lawson
associated families and proves the finiteness theorem under a prescribed
Willmore energy bound, Theorem~\ref{thm:finiteness}.

\section{Global Helicoidal Minimal Immersions}
\label{sec:global_immersions}

This section recalls the local helicoidal parametrization and the
reconstruction equations of Castro et al.~\cite{Castro2024}. For
$0<c<1/2$, we reconstruct the profile curve globally, compute its longitude
increment, and define the resulting helicoidal immersion. The case $c=0$ is
described separately by the explicit Lawson parametrization.

Throughout the paper, we adopt the normalization of Castro et al., in which
the angular speed in the second complex coordinate is normalized to one and
the pitch $h$ is the coefficient of $t$ in the first complex coordinate.

\subsection{Local helicoidal parametrizations}

We first recall how a local helicoidal parametrization is obtained from
a given unit-speed profile curve. The profile curves relevant to the
present family will be reconstructed globally in the next subsection.

We identify
\[
\mathbb S^3
=
\{(x_1,x_2,x_3,x_4)\in\mathbb R^4:
x_1^2+x_2^2+x_3^2+x_4^2=1\}
=
\{(\omega_1,\omega_2)\in\mathbb C^2:
|\omega_1|^2+|\omega_2|^2=1\}.
\]
The axis is the great circle
\[
c_0=\{(x_1,x_2,0,0)\in\mathbb S^3\}.
\]
In this normalization, a helicoidal surface with pitch $h\geq0$ is locally
written as
\[
X^h(\xi)(s,t)
=
\bigl(
x(s)\cos(ht)-y(s)\sin(ht),
x(s)\sin(ht)+y(s)\cos(ht),
z(s)\cos t,
z(s)\sin t
\bigr),
\]
where
\[
\xi(s)=(x(s),y(s),z(s),0)
\]
is a unit-speed curve in
\[
\mathbb S^2_+
=
\{(x_1,x_2,x_3,0)\in\mathbb S^3:x_3>0\}.
\]
On any subinterval on which $z(s)<1$, the same profile curve may be
written in latitude--longitude coordinates as
\[
\xi(s)
=
\bigl(
\cos\varphi(s)\cos\lambda(s),
\cos\varphi(s)\sin\lambda(s),
\sin\varphi(s),
0
\bigr).
\]
The corresponding helicoidal parametrization is
\cite[Eq.~(2.4)]{Castro2024}
\begin{equation}\label{eq:castro_local_helicoidal_parametrization}
X^h(\xi)(s,t)
=
\bigl(
\cos\varphi(s)e^{i(ht+\lambda(s))},
\sin\varphi(s)e^{it}
\bigr).
\end{equation}
Here
\[
z(s)=\sin\varphi(s),
\qquad
0<\varphi(s)<\frac{\pi}{2}.
\]
Thus $\varphi(s)$ is the geodesic distance in $\mathbb S^2$ from
$\xi(s)$ to the great circle $c_0$, while $z(s)$ is the sine of this
distance. 

For a unit-speed spherical curve
$\xi(s)=(x(s),y(s),z(s))$, with the fourth coordinate suppressed and
dots denoting derivatives with respect to $s$, its spherical angular
momentum with respect to $c_0$ is \cite[Eq.~(3.2)]{Castro2024}
\begin{equation}\label{eq:spherical_angular_momentum}
\mathcal K(s)
=
-\det(\xi(s),\dot\xi(s),\mathbf e_3)
=
\dot x(s)y(s)-x(s)\dot y(s),
\qquad
\mathbf e_3=(0,0,1).
\end{equation}
Substituting the latitude--longitude expression for $\xi$ into
\eqref{eq:spherical_angular_momentum} gives
\begin{equation}\label{eq:angular_momentum_longitude}
\mathcal K(s)
=
-\bigl(1-z(s)^2\bigr)\dot\lambda(s).
\end{equation}

By \cite[Theorem~3.1]{Castro2024}, if $z(s)$ is nonconstant
and the angular momentum along the profile curve can be expressed as a function
of $z(s)$, then this function determines the profile curve uniquely up to
rotations about the $z$-axis. More precisely, the proof gives the relations
\begin{equation}\label{eq:Castro_reconstruction}
\dot z^{\,2}
=
1-z^2-\mathcal K(z)^2,
\qquad
\dot\lambda
=
\frac{\mathcal K(z)}{z^2-1}.
\end{equation}
For the minimal helicoidal surfaces with $0<c<1/2$,
Castro et al.'s Corollary~5.1 gives the spherical angular momentum
\begin{equation}\label{eq:Kc}
\mathcal K_c^h(z)
=
-\frac{c\sqrt{h^2+(1-h^2)z^2}}{\sqrt{z^4+h^2c^2}},
\qquad
h\geq0,\quad 0<c<\frac12.
\end{equation}
With the sign chosen in \eqref{eq:Kc},
$\mathcal K_c^h(z)<0$ for every value of $z$ attained by the profile.
We also record the induced metric of the surface, which will be used below to verify
that the global parametrization remains an immersion at the critical points
of the height function and later to compute its area density. By
\cite[Eqs.~(4.1) and~(4.3)]{Castro2024}, it is
\begin{equation}\label{eq:local_induced_metric}
    g_{11}=1,\qquad
    g_{12}=-h\mathcal K,\qquad
    g_{22}=h^2+(1-h^2)z^2,
\end{equation}
with $\mathcal K=\mathcal K_c^h(z)$ as in \eqref{eq:Kc}, $(s,t)$ the ordered coordinates.

For $c=0$, the height function attains its minimum value $z=0$, where
the latitude--longitude coordinates degenerate as the profile curve
intersects the rotation axis. We therefore restrict the reconstruction below
to $0<c<1/2$ and treat $c=0$ separately using the Lawson parametrization.

\subsection{Global reconstruction of the profile curve}

Assume throughout this subsection that $0<c<1/2$. The two reconstruction
equations in \eqref{eq:Castro_reconstruction} will be used successively.
The first is used to construct the height function $z(s)$ as a smooth periodic
function. Once the height function has been obtained globally, the second
determines the longitude $\lambda(s)$ up to an additive constant. We then construct the
global unit-speed profile curve and compute the longitude increment over one
period of the height function.

For the spherical angular momentum $\mathcal K_c^h(z)$ defined in \eqref{eq:Kc}, set
\begin{equation}\label{eq:Fhc_definition}
F_{c}^{h}(z)
:=
1-z^2-\bigl(\mathcal K_c^h(z)\bigr)^2.
\end{equation}
Substituting $\mathcal K=\mathcal K_c^h$ into the first reconstruction
equation in \eqref{eq:Castro_reconstruction}, we obtain
\[
\dot z^{\,2}=F_c^h(z).
\]
The factorization below identifies the minimum and maximum values of the
height function. We then construct the increasing and decreasing inverse
branches, reflect them at their endpoints, and extend the resulting
function periodically to obtain a smooth solution $z(s)$ on $s\in\mathbb R$.

\begin{lem}\label{lem:height_range_factorization}
Let $h\geq0$ and $0<c<1/2$, and set
\begin{equation}\label{eq:height_extrema}
z_{\min}
=
\sqrt{\frac{1-\sqrt{1-4c^2}}{2}},
\qquad
z_{\max}
=
\sqrt{\frac{1+\sqrt{1-4c^2}}{2}}.
\end{equation}
For $z>0$,
\begin{equation}\label{eq:Fhc_factorization}
F_{c}^{h}(z)
=
\frac{z^2(z^2-z_{\min}^2)(z_{\max}^2-z^2)}
     {z^4+h^2c^2}.
\end{equation}
Consequently, $z_{\min}$ and $z_{\max}$ are the only positive zeros
of $F_{c}^{h}$, are simple and independent of $h$, and satisfy
$0<z_{\min}<z_{\max}<1$. Moreover,
\[
F_{c}^{h}>0
\quad\text{on }(z_{\min},z_{\max}),
\qquad
\frac{1}{\sqrt{F_{c}^{h}}}
\in L^1([z_{\min},z_{\max}]).
\]
\end{lem}

\begin{proof}
Substituting \eqref{eq:Kc} into $F_{c}^{h}(z)=0$ gives
\[
1-z^2-\frac{c^2\bigl(h^2+(1-h^2)z^2\bigr)}
{z^4+h^2c^2}
=0.
\]
After clearing the denominator and multiplying by $-1$, the equation is
equivalent to
\[
z^2(z^4-z^2+c^2)=0.
\]
Since the profile lies in $\mathbb S^2_+$, we have $z>0$, and hence
the factor $z^2$ may be discarded. Thus
\[
F_{c}^{h}(z)=0
\quad\Longleftrightarrow\quad
z^4-z^2+c^2=0.
\]
Solving this quadratic equation in $z^2$ gives
\[
z^2=\frac{1\pm\sqrt{1-4c^2}}{2}.
\]
Therefore the two positive zeros of $F_{c}^{h}$ are
\[
z_{\min}
 =\sqrt{\frac{1-\sqrt{1-4c^2}}{2}},
\qquad
z_{\max}
 =\sqrt{\frac{1+\sqrt{1-4c^2}}{2}},
\]
as recorded in \eqref{eq:height_extrema}.

Using
\[
z^4-z^2+c^2=(z^2-z_{\min}^2)(z^2-z_{\max}^2),
\]
we obtain the factorization \eqref{eq:Fhc_factorization}.
Since the denominator is strictly positive on
$[z_{\min},z_{\max}]$, the factorization shows that
$F_{c}^{h}>0$ on $(z_{\min},z_{\max})$ and that
$z_{\min}$ and $z_{\max}$ are simple zeros. Consequently,
\[
\frac{1}{\sqrt{F_{c}^{h}}}\in
L^1\bigl([z_{\min},z_{\max}]\bigr).
\]
\end{proof}

To construct a smooth periodic solution by reflecting the inverse branches
at the endpoint values, we use the following elementary extension lemma.

\begin{lem}\label{lem:smooth_even_extension}
Let $F$ be smooth near $a$, and assume that on a right-hand neighborhood
of $a$ one has $F(u)\geq0$ and
$F(u)=(u-a)G(u)$ for a smooth function $G$ with $G(a)>0$. If $u(s)\geq a$ is defined for
$s\geq0$ by
\[
    s=\int_a^{u(s)}\frac{\dif v}{\sqrt{F(v)}},
\]
then, after possibly shrinking the interval, there is a smooth function
$\psi\geq a$ with $\psi(0)=a$ such that $u(s)=\psi(s^2)$. In particular, the
branch extends smoothly and evenly across $s=0$.

Similarly, let $F$ be smooth near $b$, and assume that on a left-hand
neighborhood of $b$ one has $F(u)\geq0$ and
$F(u)=(b-u)G(u)$ for a smooth function $G$ with $G(b)>0$. If $u(s)\leq b$ is defined for
$s\leq0$ by
\[
    -s=\int_{u(s)}^b\frac{\dif v}{\sqrt{F(v)}},
\]
then there is a smooth function $\psi\leq b$ with $\psi(0)=b$ such that
$u(s)=\psi(s^2)$. Equivalently, one may write
$u(s)=b-\eta(s^2)$, where $\eta\geq 0$ is smooth and $\eta(0)=0$.
\end{lem}

\begin{proof}
For the lower endpoint $a$, put $y=\sqrt{u-a}$.  Then
\[
    s=2\int_0^y \frac{\dif r}{\sqrt{G(a+r^2)}}=yH(y^2),
\]
where $H$ is smooth and $H(0)=2/\sqrt{G(a)}>0$.  The map
$y\mapsto yH(y^2)$ has nonzero derivative at $0$, so it has a smooth
local inverse. Since $yH(y^2)$ is odd, its local inverse is also odd and has
the form $y=sJ(s^2)$, with $J$ smooth. Hence
$u-a=y^2=s^2J(s^2)^2$, which gives the asserted $\psi$.

For the upper endpoint $b$, put $y=\sqrt{b-u}$.  Then
\[
    -s=2\int_0^y\frac{\dif r}{\sqrt{G(b-r^2)}}=y\widetilde H(y^2),
\]
with $\widetilde H(0)=2/\sqrt{G(b)}>0$. The same argument, using the
oddness of the local inverse, gives $y=(-s)\widetilde J(s^2)$, and hence
\[
    b-u=s^2\widetilde J(s^2)^2.
\]
This proves the asserted form.
\end{proof}

By Lemma~\ref{lem:height_range_factorization}, together with
\eqref{eq:height_extrema} and \eqref{eq:Fhc_factorization}, $F_{c}^{h}>0$ on
$(z_{\min},z_{\max})$, and its only zeros on
$[z_{\min},z_{\max}]$ are the two endpoints.
Set
\begin{equation}\label{eq:half_and_full_height_periods}
\ell
=
\int_{z_{\min}}^{z_{\max}}
\frac{\dif u}{\sqrt{F_{c}^{h}(u)}},
\qquad
L=2\ell.
\end{equation}
The finiteness of $\ell$ follows from
Lemma~\ref{lem:height_range_factorization}.

On the increasing branch, define $z:[0,\ell]\to[z_{\min},z_{\max}]$ implicitly by
\begin{equation}\label{eq:z_first_branch}
s=
\int_{z_{\min}}^{z(s)}
\frac{\dif u}{\sqrt{F_{c}^{h}(u)}},
\qquad
0\le s\le \ell.
\end{equation}
Then $z(0)=z_{\min}$, $z(\ell)=z_{\max}$, and
\[
\dot z=\sqrt{F_{c}^{h}(z)}>0
\qquad\text{on }(0,\ell).
\]
On the decreasing branch, define $z:[\ell,L]\to[z_{\min},z_{\max}]$ by
\begin{equation}\label{eq:z_second_branch}
s-\ell
=
\int_{z(s)}^{z_{\max}}
\frac{\dif u}{\sqrt{F_{c}^{h}(u)}},
\qquad
\ell\le s\le L.
\end{equation}
Then $z(\ell)=z_{\max}$, $z(L)=z_{\min}$, and
\[
\dot z=-\sqrt{F_{c}^{h}(z)}<0
\qquad\text{on }(\ell,L).
\]

Since the zeros of $F_c^h$ at $z_{\min}$ and $z_{\max}$ are simple,
Lemma~\ref{lem:smooth_even_extension} shows that the two branches join
smoothly at $s=\ell$ and extend evenly across the endpoint values. By the construction in
\eqref{eq:z_first_branch}--\eqref{eq:z_second_branch}, we have the symmetry
\begin{equation}\label{eq:z_symmetry}
z(\ell+s)=z(\ell-s)\quad (|s|\leq\ell),
\qquad
z(L-s)=z(s)\quad (0\leq s\leq L).
\end{equation}
This, together with the smooth even extensions
at $s=0$ and $s=L$, shows that these extensions agree after translation
by $L$. Hence periodic repetition defines a smooth global height function
\begin{equation}\label{eq:z_periodic_global}
\begin{array}{cc}
 z_c^h:\mathbb R\longrightarrow[z_{\min},z_{\max}] & \text{satisfying} \\
  z_c^h(s+L)=z_c^h(s), & \bigl(\dot z_c^h(s)\bigr)^2=F_c^h\bigl(z_c^h(s)\bigr) 
\end{array}
\end{equation}
for all $s\in\mathbb R$.
Its critical points are precisely $s=k\ell$, $k\in\mathbb Z$. More
explicitly,
\[
z_c^h(2k\ell)=z_{\min},
\qquad
z_c^h((2k+1)\ell)=z_{\max},
\]
so $s=2k\ell$ are minimum points, and
$s=(2k+1)\ell$ are maximum points. Moreover, $z_c^h(s)$ is strictly increasing on $[2k\ell, (2k+1)\ell]$ and strictly decreasing on $[(2k+1)\ell, (2k+2)\ell]$ for $k\in\mathbb Z$.
Define the global latitude function by
\begin{equation}\label{eq:varphi_definition}
\varphi_c^h(s)=\arcsin z_c^h(s), \quad s\in\mathbb{R}.
\end{equation}
Since $[z_{\min},z_{\max}]\subset(0,1)$, the function
$\varphi_c^h$ is smooth and $L$-periodic.

With the global height function fixed, the second reconstruction equation in
\eqref{eq:Castro_reconstruction} determines the longitude up to an additive
constant. We fix this constant by defining the longitude to be 
\begin{equation}\label{eq:lambda_integral}
\lambda_c^h(s)
=
\int_0^s
\frac{\mathcal K_c^h(z_c^h(r))}{z_c^h(r)^2-1}\,\dif r.
\end{equation}
Since
\[
0<z_{\min}\le z_c^h(r)\le z_{\max}<1,
\]
the integrand is smooth on $\mathbb R$. Hence
$\lambda_c^h\in C^\infty(\mathbb R)$,
$\lambda_c^h(0)=0$, and
\begin{equation}\label{eq:lambda_derivative}
\dot\lambda_c^h(s)
=
\frac{\mathcal K_c^h(z_c^h(s))}{z_c^h(s)^2-1}
>0.
\end{equation}
Thus $\lambda_c^h$ is strictly increasing on $\mathbb R$.
Every other solution of the longitude equation differs from
$\lambda_c^h$ by an additive constant. Replacing $\lambda_c^h$ by
$\lambda_c^h+a$ amounts to composing the resulting profile curve with the
ambient isometry
\[
(\omega_1,\omega_2)\longmapsto(e^{ia}\omega_1,\omega_2).
\]
Thus the resulting surface is determined up to congruence.

Define the global profile curve
\begin{equation}\label{eq:global_profile_curve}
\begin{aligned}
\xi_c^h:\mathbb R&\longrightarrow\mathbb S^2_+,\\
\xi_c^h(s)
&=
\bigl(
\cos\varphi_c^h(s)\cos\lambda_c^h(s),
\cos\varphi_c^h(s)\sin\lambda_c^h(s),
\sin\varphi_c^h(s),
0
\bigr).
\end{aligned}
\end{equation}
The curve $\xi_c^h$ is smooth on $\mathbb R$ and has unit speed. Indeed,
using \eqref{eq:z_periodic_global} and \eqref{eq:lambda_derivative}, we find
\[
\begin{aligned}
\lvert\dot\xi_c^h(s)\rvert^2
&=
\frac{\dot z_c^h(s)^2}{1-z_c^h(s)^2}
+
\bigl(1-z_c^h(s)^2\bigr)\dot\lambda_c^h(s)^2\\
&=
\frac{
1-z_c^h(s)^2-\bigl(\mathcal K_c^h(z_c^h(s))\bigr)^2
+\bigl(\mathcal K_c^h(z_c^h(s))\bigr)^2
}{1-z_c^h(s)^2}\\
&=1.
\end{aligned}
\]
Thus \eqref{eq:global_profile_curve} is a globally defined unit-speed
profile curve.

Although $z_c^h$ and $\varphi_c^h$ are $L$-periodic, the longitude
$\lambda_c^h$ is strictly increasing on $\mathbb R$. Its increment over one period of the height function is the quantity
entering the closing condition in Section~\ref{sec:compactness}.

\begin{lem}\label{lem:longitude_increment}
Let $h\geq0$ and $0<c<1/2$. Let $\lambda=\lambda_c^h$ be defined by
\eqref{eq:lambda_integral}, and let $L=2\ell$ be defined by
\eqref{eq:half_and_full_height_periods}. Then
$\lambda(s+L)-\lambda(s)$ is independent of $s$, and
\begin{equation}\label{eq:longitude_drift}
\lambda(s+L)-\lambda(s)
=
2\int_{z_{\min}}^{z_{\max}}
\frac{-\mathcal K_c^h(z)}
     {(1-z^2)\sqrt{F_c^h(z)}}\,\dif z
\qquad
\text{for all }s\in\mathbb R,
\end{equation}
where $\mathcal K_c^h$, $F_c^h$, and $z_{\min},z_{\max}$ are defined
in \eqref{eq:Kc}, \eqref{eq:Fhc_definition}, and
\eqref{eq:height_extrema}, respectively.
\end{lem}

\begin{proof}
By \eqref{eq:lambda_derivative} and the $L$-periodicity of $z_c^h$,
the derivative $\dot\lambda$ is $L$-periodic; hence
$\lambda(s+L)-\lambda(s)$ is independent of $s$. Taking $s=0$,
splitting the integral \eqref{eq:lambda_integral} for $\lambda(L)$ at $\ell$, and using the symmetry of $z_c^h$ \eqref{eq:z_symmetry}, we get
\eqref{eq:longitude_drift}.
\end{proof}

\subsection{Global immersions and minimality}

\begin{defn}
\label{def:global_parametrizations}
Let $h\geq0$ and $0\le c<1/2$. We define the global helicoidal parametrization
$X_c^h$ and the corresponding helicoidal surface
$\operatorname{Hel}_c^h$ as follows.

Assume first that $0<c<1/2$. Substituting the profile curve
\eqref{eq:global_profile_curve} into the helicoidal construction
\eqref{eq:castro_local_helicoidal_parametrization}, define
\begin{equation}\label{eq:global_positive_c_parametrization}
\begin{aligned}
X_c^h:=X^h(\xi_c^h):\mathbb R^2&\longrightarrow\mathbb S^3,\\
X_c^h(s,t)
&=
\bigl(
\cos\varphi_c^h(s)e^{i(ht+\lambda_c^h(s))},
\sin\varphi_c^h(s)e^{it}
\bigr).
\end{aligned}
\end{equation}
We write $\operatorname{Hel}_c^h$ for the helicoidal surface represented by
$X_c^h$.

For $c=0$ and $h>0$, define
\begin{equation}\label{eq:global_lawson_parametrization}
\begin{aligned}
X_0^h:\mathbb R^2&\longrightarrow\mathbb S^3,\\
X_0^h(s,t)
&=
\bigl(
\cos s\,e^{iht},
\sin s\,e^{it}
\bigr),
\end{aligned}
\end{equation}
and write $\operatorname{Hel}_0^h$ for the Lawson spherical helicoid
represented by $X_0^h$.

At $(h,c)=(0,0)$, define
\[
\begin{aligned}
X_0^0:\mathbb R^2&\longrightarrow\mathbb S^3,\\
X_0^0(s,t)&=\bigl(\cos s,\sin s\,e^{it}\bigr),
\end{aligned}
\]
and set
\[
\operatorname{Hel}_0^0=X_0^0(\mathbb R^2)=\mathbb S^2,
\]
the totally geodesic sphere. The map $X_0^0$ is the usual polar
parametrization and fails to be an immersion along $s\in\pi\mathbb Z$.
\end{defn}

\begin{prop}
\label{prop:global_minimal_immersion}
Let $h\geq0$ and $0\le c<1/2$, with $(h,c)\neq(0,0)$. Then the map
\[
X_c^h:\mathbb R^2\longrightarrow\mathbb S^3
\]
defined by \eqref{eq:global_positive_c_parametrization} when $c>0$
and by \eqref{eq:global_lawson_parametrization} when $c=0$ is a smooth
helicoidal minimal immersion. Consequently, $\operatorname{Hel}_c^h$ is a helicoidal
minimal surface.
\end{prop}

\begin{proof}
Assume first that $0<c<1/2$. The curve $\xi_c^h$ defined in
\eqref{eq:global_profile_curve} is smooth and unit-speed, and
$z=z_c^h$ is $L$-periodic with values in
$[z_{\min},z_{\max}]\subset(0,1)$. By
\eqref{eq:angular_momentum_longitude} and
\eqref{eq:lambda_derivative}, its spherical angular momentum satisfies
\[
\mathcal K_{\xi_c^h}(s)
=-(1-z_c^h(s)^2)\dot\lambda_c^h(s)
=\mathcal K_c^h(z_c^h(s))
\]
on all of $\mathbb R$. Hence $X_c^h$ is globally defined and smooth.
For the remainder of the proof, write $z=z_c^h$ and
$\lambda=\lambda_c^h$.

By direct computation, equivalently by
\eqref{eq:local_induced_metric}, the determinant of the induced metric is
\[
\det g
=
\frac{z^4\bigl(h^2+(1-h^2)z^2\bigr)}{z^4+h^2c^2}>0.
\]
Thus $X_c^h$ is an immersion on all of $\mathbb R^2$, including the
parameter lines $s=k\ell$ on which $\dot z=0$.

On every connected component of $\{\dot z\neq0\}$, the functions $z$,
$\lambda$, and $\mathcal K_\xi=\mathcal K_c^h(z)$ satisfy the local
reconstruction equations of Castro et al. For a helicoidal surface with
spherical angular momentum $\mathcal K=\mathcal K(z)$, their
mean-curvature formula is
\[
2H(z)=zA'(z)+2A(z),
\qquad
A(z)
=
\frac{\mathcal K(z)}
{\sqrt{h^2(1-\mathcal K(z)^2)+(1-h^2)z^2}}
\]
\cite[Eq.~(4.12)]{Castro2024}. Substitution of
$\mathcal K=\mathcal K_c^h$ gives $A(z)=-c/z^2$, and hence $H=0$ on
$\{\dot z\neq0\}\times\mathbb R$. This set is dense because the
critical points of $z$ are precisely $s=k\ell$. Since the
mean-curvature vector of the global immersion is continuous, it vanishes
on all of $\mathbb R^2$. Therefore $X_c^h$ is a smooth minimal
immersion.

Now let $c=0$ and $h>0$. For the explicit map
\eqref{eq:global_lawson_parametrization}, one has
\[
(X_0^h)^*g_{\mathbb S^3}
=
(\dif s)^2+\bigl(h^2\cos^2s+\sin^2s\bigr)(\dif t)^2,
\]
which is positive definite because $h>0$. Thus $X_0^h$ is an immersion.
It parametrizes a Lawson spherical helicoid. Its minimality is
classical \cite[Section~7]{Lawson1970} and also follows from the helicoidal
classification in \cite[Example~2.3]{Castro2024}.
\end{proof}

\section{Compactness Criterion}
\label{sec:compactness}

This section characterizes compactness for the global parametrization
constructed in Section~\ref{sec:global_immersions}. We first specify the
notion of compactness, then extract the relevant period
data and determine the local symmetries of the parametrization. A
developing-map argument then yields the compactness criterion. Finally, we
study the range of the normalized longitude increment.

\begin{defn}\label{def:compactness}
Let $(h,c)\neq(0,0)$. We say that
$\operatorname{Hel}_c^h$ is \emph{compact} if there exist a compact
connected smooth surface $M$ without boundary and a smooth immersion
\[
f:M\longrightarrow\mathbb S^3
\]
such that, for every $p\in M$, there exist open neighborhoods
$p\in U_p\subset M$ and $V_p\subset\mathbb R^2$, together with a
diffeomorphism
\[
\Phi_p:U_p\longrightarrow V_p
\]
satisfying
\[
f|_{U_p}=X_c^h\circ\Phi_p.
\]
The pair $(M,f)$ is called a \emph{compact realization} of
$\operatorname{Hel}_c^h$.

For $(h,c)=(0,0)$, we set
$\operatorname{Hel}_0^0=\mathbb S^2$, which is compact.
\end{defn}

\subsection{Period data and local symmetries}

For any smooth map $X:\mathbb R^2\to N$, where $N$ is a smooth
manifold, define its period group by
\begin{equation}\label{eq:translation_period_definition_general}
\operatorname{Per}(X)
:=
\left\{
 v\in\mathbb R^2:
 X(u+v)=X(u)\ \text{for every }u\in\mathbb R^2
\right\}.
\end{equation}
A nonzero vector of $\operatorname{Per}(X)$ is called a \emph{period vector}
of $X$.

We first consider the case $0<c<1/2$. For the local
symmetry analysis, we need the precise level sets of the height function, not
only its periodicity. The following lemma also identifies its minimal positive
period.

\begin{lem}\label{lem:minimal_height_period}
Let $h\geq0$ and $0<c<1/2$. Let $z=z_c^h$ be the global height function defined
by \eqref{eq:z_periodic_global}, let
$\varphi=\varphi_c^h$ be defined by \eqref{eq:varphi_definition}, and
let $L=2\ell$ be defined by
\eqref{eq:half_and_full_height_periods}. Then, for any
$r,s\in\mathbb R$,
\begin{equation}\label{eq:z_fiber_structure}
z(r)=z(s)
\quad\Longleftrightarrow\quad
r-s\in L\mathbb Z
\ \text{or}\
r+s\in L\mathbb Z.
\end{equation}
The same equivalence holds with $z$ replaced by $\varphi$. In
particular, $L$ is the minimal positive period of both $z$ and
$\varphi$; more precisely,
\[
z(s+P)=z(s)\quad\text{for every }s\in\mathbb R
\]
holds if and only if $P\in L\mathbb Z$.
\end{lem}

\begin{proof}
Recall from \eqref{eq:Fhc_definition} and
\eqref{eq:height_extrema} the function $F_c^h$ and the extremal values
$z_{\min},z_{\max}$ of the height function. Define
\[
\sigma(u)
=
\int_{z_{\min}}^{u}
\frac{\dif v}{\sqrt{F_c^h(v)}},
\qquad
z_{\min}\le u\le z_{\max}.
\]
Then $\sigma$ is strictly increasing from
$[z_{\min},z_{\max}]$ onto $[0,\ell]$. By
\eqref{eq:z_first_branch}, the unique point of the increasing branch
$[0,\ell]$ at which $z=u$ is $\sigma(u)$. Moreover,
\[
\int_u^{z_{\max}}
\frac{\dif v}{\sqrt{F_c^h(v)}}
=
\ell-\sigma(u),
\]
so \eqref{eq:z_second_branch} shows that the corresponding point of the
decreasing branch $[\ell,L]$ is
\[
L-\sigma(u).
\]
Consequently, modulo $L$, the points at which $z=u$ are represented by
$\sigma(u)$ and $-\sigma(u)$. Hence, by the $L$-periodicity of $z$,
\[
z(r)=z(s)
\quad\Longleftrightarrow\quad
r\equiv s\pmod L
\ \text{or}\
r\equiv -s\pmod L,
\]
which is exactly \eqref{eq:z_fiber_structure}.

Since $\arcsin$ is injective on $(0,1)$, we have
\[
z(r)=z(s)
\quad\Longleftrightarrow\quad
\varphi(r)=\varphi(s)
\]
for all $r,s\in\mathbb R$. Hence $z$ and $\varphi$ have the same period
group. Every element of $L\mathbb Z$ is a period by
\eqref{eq:z_periodic_global}. Conversely, if $P$ is a period of $z$,
then
\[
z(P)=z(0),
\]
and \eqref{eq:z_fiber_structure} with $r=P$ and $s=0$ yields
$P\in L\mathbb Z$. Therefore, the period group of both $z$ and
$\varphi$ is $L\mathbb Z$, and $L$ is their minimal positive period.
\end{proof}

We now normalize the longitude change over this minimal period.

\begin{defn}\label{def:q}
Let $h\geq0$ and $0<c<1/2$. Let $\lambda_c^h$ be the normalized longitude defined by
\eqref{eq:lambda_integral}, and let $L$ be the minimal positive period of
$z_c^h$ given by Lemma~\ref{lem:minimal_height_period}. The
\emph{normalized longitude increment} is
\begin{equation}\label{eq:q_definition}
q(h,c)
:=
\frac{\lambda_c^h(s+L)-\lambda_c^h(s)}{2\pi}.
\end{equation}
By Lemma~\ref{lem:longitude_increment}, this definition is independent
of $s$.
\end{defn}

Iterating \eqref{eq:q_definition}, also backwards when $m<0$, gives
\begin{equation}\label{eq:longitude_quasiperiodicity}
\lambda_c^h(s+mL)=\lambda_c^h(s)+2\pi m q(h,c)
\qquad\text{for every }m\in\mathbb Z.
\end{equation}
The translation relations associated with the height period $L$ and the
angular period $2\pi$ follow from
\eqref{eq:global_positive_c_parametrization} and
\eqref{eq:longitude_quasiperiodicity}. By \eqref{eq:varphi_definition} and
Lemma~\ref{lem:minimal_height_period}, $\varphi_c^h$ is $L$-periodic;
hence, for every $m,n\in\mathbb Z$,
\begin{equation*}
\begin{aligned}
X_c^h(s+mL,t+2\pi n)
=
\Bigl(
e^{2\pi i(mq(h,c)+nh)}
 \cos\varphi_c^h(s)e^{i(ht+\lambda_c^h(s))},\quad
\sin\varphi_c^h(s)e^{it}
\Bigr).
\end{aligned}
\end{equation*}
Since $\cos\varphi_c^h(s)>0$, it follows that
\begin{equation}\label{eq:global_translation_period_condition}
X_c^h(s+mL,t+2\pi n)=X_c^h(s,t)
\quad\Longleftrightarrow\quad
mq(h,c)+nh\in\mathbb Z.
\end{equation}

The next lemma uses the longitude equation and the fiber
characterization \eqref{eq:z_fiber_structure} of the height function to
show that these translations exhaust the local symmetries of $X_c^h$.

\begin{lem}
\label{lem:local_transition_rigidity}
Let $h\geq0$ and $0<c<1/2$, and let
\[
X=X_c^h
\]
be the global parametrization defined in
\eqref{eq:global_positive_c_parametrization}. Let $L$ be the minimal
positive period of the height function given by
Lemma~\ref{lem:minimal_height_period}, and let $q(h,c)$ be defined by
\eqref{eq:q_definition}. If
\[
F=(S,T):\Omega\longrightarrow\Omega'
\]
is a $C^1$-diffeomorphism between nonempty connected open subsets of
$\mathbb R^2$ such that
\[
X\circ F=X,
\]
then there exist integers $m,n\in\mathbb Z$ such that
\[
F(s,t)=(s+mL,t+2\pi n)
\qquad\text{for}\quad (s,t)\in\Omega.
\]
Moreover, the integers $m$ and $n$ satisfy
\begin{equation}\label{eq:positive_c_period_condition}
mq(h,c)+nh\in\mathbb Z.
\end{equation}
Conversely, if $m,n\in\mathbb Z$ satisfy
\eqref{eq:positive_c_period_condition}, then
\[
X(s+mL,t+2\pi n)=X(s,t)
\qquad\text{for all }(s,t)\in\mathbb R^2.
\]
\end{lem}

\begin{proof}
Write $F(s,t)=(S(s,t),T(s,t))$. Comparison of the second complex components in $X\circ F=X$ gives
\[
\sin\varphi(S(s,t))e^{iT(s,t)}
=
\sin\varphi(s)e^{it}.
\]
Since
\[
0<z_{\min}\leq \sin\varphi(s)\leq z_{\max}<1,
\]
comparison of norms and arguments yields
\begin{equation}\label{eq:phi_in_rigidity_proof}
\varphi(S(s,t))=\varphi(s),
\qquad
T(s,t)-t\in2\pi\mathbb Z.
\end{equation}
The continuous function $(T-t)/(2\pi)$ is integer-valued and hence
constant on the connected set $\Omega$. Therefore, for some
$n\in\mathbb Z$,
\begin{equation}\label{eq:T_translation_in_rigidity_proof}
T(s,t)=t+2\pi n
\qquad\text{on }\Omega.
\end{equation}

It follows from Lemma~\ref{lem:minimal_height_period}, the continuity of $S$, and \eqref{eq:phi_in_rigidity_proof} that
\[
S(s,t)=\pm s+mL
\qquad\text{on }\Omega.
\]
Comparing the first complex components and using
\eqref{eq:phi_in_rigidity_proof} and
\eqref{eq:T_translation_in_rigidity_proof} gives
\[
\lambda(S(s,t))-\lambda(s)+2\pi hn\in2\pi\mathbb Z.
\]
Since this expression is continuous on the connected set $\Omega$, it
is constant. Hence
\begin{equation}\label{eq:lambda_constant_in_rigidity_proof}
\lambda(S(s,t))-\lambda(s)=C
\end{equation}
for some $C\in\mathbb R$.
 \eqref{eq:phi_in_rigidity_proof} implies $z(S)=z(s)$. Since
the right-hand side of \eqref{eq:lambda_derivative} depends only on
$z$, we have
\[
\dot\lambda(S)=\dot\lambda(s).
\]
Differentiating \eqref{eq:lambda_constant_in_rigidity_proof} with
respect to $s$ therefore gives
\[
0
=
\dot\lambda(S)S_s-\dot\lambda(s)
=
\dot\lambda(s)(S_s-1).
\]
Thus $S_s=1$ since $\dot\lambda(s)>0$, proving 
\begin{equation}\label{eq:S_translation_in_rigidity_proof}
S(s,t)=s+mL
\qquad\text{on }\Omega.
\end{equation}

Finally, substituting \eqref{eq:T_translation_in_rigidity_proof} and
\eqref{eq:S_translation_in_rigidity_proof} into the first complex
component of $X\circ F=X$, and using $\cos\varphi_c^h(s)>0$, gives
\[
e^{2\pi i(mq(h,c)+nh)}=1.
\]
Therefore
\[
mq(h,c)+nh\in\mathbb Z.
\]
Conversely, \eqref{eq:global_translation_period_condition} shows that
any pair $m,n\in\mathbb Z$ satisfying this condition defines a global
symmetry of $X$.
\end{proof}

We next turn to the Lawson case $c=0$. For $h>0$, recall from \eqref{eq:global_lawson_parametrization} that
\[
X_0^h(s,t)=\bigl(\cos s\,e^{iht},\sin s\,e^{it}\bigr).
\]
The following lemma determines all of its local symmetries.

\begin{lem}
\label{lem:lawson_complete_local_symmetry}
Let $h>0$, and let
$
\Phi=(S,T):\Omega\to\Omega'
$
be a $C^1$-diffeomorphism between nonempty connected open subsets of
$\mathbb R^2$. Then
\[
X_0^h\circ\Phi=X_0^h
\]
if and only if there exist
$
\varepsilon\in\{1,-1\}
$
and $k,m\in\mathbb Z$ such that
\begin{equation}\label{eq:lawson_complete_local_symmetry_form}
\Phi(s,t)=\bigl(\varepsilon s+k\pi,t+m\pi\bigr)
\end{equation}
on $\Omega$, with
\begin{equation}\label{eq:lawson_complete_local_symmetry_conditions}
m+k\in
\begin{cases}
2\mathbb Z, & \varepsilon=1,\\
2\mathbb Z+1, & \varepsilon=-1,
\end{cases}
\qquad
hm+k\in2\mathbb Z.
\end{equation}
Equivalently, when $\varepsilon=1$, the integers $m$ and $k$ are
both even or both odd; when $\varepsilon=-1$, exactly one of them is
even. The map $\Phi$ is a translation in the first case and is
orientation reversing in the second.
\end{lem}

\begin{proof}
Assume first that $X_0^h\circ\Phi=X_0^h$. Comparing the norms of the
complex components gives
\[
|\cos S|=|\cos s|,
\qquad
|\sin S|=|\sin s|.
\]
On the dense open set where $\sin s\neq0$, the second component implies
$T-t\in\pi\mathbb Z$. Since $T-t$ is continuous, the same inclusion holds
on all of $\Omega$; connectedness then gives a single $m\in\mathbb Z$ such
that
\begin{equation}\label{eq:lawson_T_global}
T=t+m\pi
\end{equation}
on $\Omega$, and hence
\begin{equation}\label{eq:lawson_sine_sign}
\sin S=(-1)^m\sin s.
\end{equation}
The first component now gives
\[
e^{i\pi hm}\cos S=\cos s.
\]
Since $\Omega$ is not contained in the zero set of $\cos s$,
$e^{i\pi hm}\in\{1,-1\}$. Thus there is $k\in\mathbb Z$ such that
\begin{equation}\label{eq:lawson_cosine_sign}
hm+k\in2\mathbb Z,
\qquad
\cos S=(-1)^k\cos s.
\end{equation}

Set
\[
\varepsilon:=(-1)^{m+k}\in\{1,-1\}.
\]
Equations \eqref{eq:lawson_sine_sign} and
\eqref{eq:lawson_cosine_sign} yield
\[
\begin{aligned}
(\cos S,\sin S)
&=
\bigl((-1)^k\cos s,(-1)^m\sin s\bigr)\\
&=
\bigl(\cos(\varepsilon s+k\pi),
      \sin(\varepsilon s+k\pi)\bigr).
\end{aligned}
\]
Therefore the function
\[
r(s,t)
:=
\frac{S(s,t)-\varepsilon s-k\pi}{2\pi}
\]
is continuous and integer-valued on the connected set $\Omega$.
Consequently it is constant; write its value as $r\in\mathbb Z$. Then
\[
S=\varepsilon s+(k+2r)\pi
\qquad\text{for}\quad (s,t)\in\Omega.
\]
Replacing $k$ by $k+2r$ preserves both the parity of $m+k$ and the
condition $hm+k\in2\mathbb Z$. Hence we may write
\[
S=\varepsilon s+k\pi.
\]
Moreover,
\[
\varepsilon=1
\iff m+k\in2\mathbb Z,
\qquad
\varepsilon=-1
\iff m+k\in2\mathbb Z+1.
\]
Together with \eqref{eq:lawson_T_global} and
\eqref{eq:lawson_cosine_sign}, this proves
\eqref{eq:lawson_complete_local_symmetry_form}--
\eqref{eq:lawson_complete_local_symmetry_conditions}.

Conversely, direct substitution of
\eqref{eq:lawson_complete_local_symmetry_form} and
\eqref{eq:lawson_complete_local_symmetry_conditions} into
\eqref{eq:global_lawson_parametrization} gives
$X_0^h\circ\Phi=X_0^h$.
\end{proof}

\begin{coro}\label{cor:lawson_irrational_rigidity}
Let $h>0$ be irrational. If
\[
\Phi:\Omega\longrightarrow\Omega'
\]
is a $C^1$-diffeomorphism between nonempty connected open subsets of
$\mathbb R^2$ such that $X_0^h\circ\Phi=X_0^h$, then
\[
\Phi(s,t)=(s+2\pi r,t)
\]
for some $r\in\mathbb Z$.
\end{coro}

\begin{proof}
Apply Lemma~\ref{lem:lawson_complete_local_symmetry}. Since $h$ is
irrational, the condition $hm+k\in2\mathbb Z$ forces $m=0$
and thus $k\in2\mathbb Z$. Now $m+k\in2\mathbb Z$, the first condition in
\eqref{eq:lawson_complete_local_symmetry_conditions} forces $\varepsilon=1$.
\end{proof}

\subsection{Developing maps and compactness}

The compactness proof proceeds by identifying the universal cover of a
compact realization with the parameter plane. Under this identification,
deck transformations become symmetries of the global parametrization, and
compactness forces the resulting translation group to have rank two. We
first record the developing-map, covering, and planar-group ingredients
needed for this argument.

\begin{lem}
\label{lem:translational_developing_map}
Let $\widetilde M$ be a connected and simply connected smooth surface,
let $N$ be a smooth manifold, and let
\[
\widetilde f:\widetilde M\longrightarrow N,
\qquad
X:\mathbb R^2\longrightarrow N
\]
be smooth maps. Suppose that $\widetilde M$ admits a covering by
nonempty connected open sets $U_\alpha$ and diffeomorphisms
\[
D_\alpha:U_\alpha\longrightarrow W_\alpha,
\qquad W_\alpha\subset\mathbb R^2\ \text{open},
\]
such that
\[
\widetilde f=X\circ D_\alpha
\qquad\text{on }U_\alpha.
\]
Assume that, on every connected component of
$U_\alpha\cap U_\beta$, the difference $D_\beta-D_\alpha$ is a
constant vector in the period group
$\operatorname{Per}(X)$ defined in
\eqref{eq:translation_period_definition_general}.

Then there exists a local diffeomorphism
\[
D:\widetilde M\longrightarrow\mathbb R^2
\]
such that $\widetilde f=X\circ D$.
\end{lem}

\begin{proof}
Since the transition maps are translations, the local
$\mathbb R^2$-valued $1$-forms $dD_\alpha$ agree on overlaps and
therefore define a global closed form
\[
\omega\in\Omega^1(\widetilde M;\mathbb R^2).
\]
Applying the usual exactness result to the two real components of
$\omega$, and using that $\widetilde M$ is simply connected, we obtain
that $\omega$ is exact. Hence
there is a smooth map
\[
D:\widetilde M\longrightarrow\mathbb R^2
\]
such that
\[
dD=\omega.
\]
Adding a constant vector, we may normalize $D$ so that
$D=D_{\alpha_0}$ on one fixed chart $U_{\alpha_0}$.

For every $\alpha$,
\[
d(D-D_\alpha)=0
\qquad\text{on }U_\alpha.
\]
Since $U_\alpha$ is connected, there is a constant vector
$c_\alpha\in\mathbb R^2$ such that
\[
D=D_\alpha+c_\alpha
\qquad\text{on }U_\alpha.
\]
On a connected component of $U_\alpha\cap U_\beta$, the identities
$D=D_\alpha+c_\alpha=D_\beta+c_\beta$ give
\[
c_\beta-c_\alpha
=-(D_\beta-D_\alpha)\in\operatorname{Per}(X).
\]
Because $\widetilde M$ is connected, $U_{\alpha_0}$ and $U_\alpha$ can be
joined by a finite chain of charts with nonempty overlaps. Applying the
preceding relation on one connected component of each overlap shows that
$c_\alpha-c_{\alpha_0}$ is a finite sum of vectors in
$\operatorname{Per}(X)$. Since $c_{\alpha_0}=0$ by the normalization and
$\operatorname{Per}(X)$ is an additive subgroup of $\mathbb R^2$, it follows
that
\[
c_\alpha\in\operatorname{Per}(X).
\]
Consequently,
\[
X\circ D
=
X\circ(D_\alpha+c_\alpha)
=
X\circ D_\alpha
=
\widetilde f
\]
on every $U_\alpha$, and hence on all of $\widetilde M$.

Since $dD=dD_\alpha$ locally and every $D_\alpha$ is a
diffeomorphism onto an open subset of $\mathbb R^2$, the map $D$ is
a local diffeomorphism.
\end{proof}

We shall use the following standard covering criterion
\cite[Chapter~VII, Lemma~3.3]{doCarmo1992}.

\begin{lem}\label{lem:surjective_covering_criterion}
Let $(M,g_M)$ be a complete Riemannian manifold, let $(N,g_N)$ be a
Riemannian manifold, and let
\[
F:M\longrightarrow N
\]
be a surjective local diffeomorphism. Suppose that
\[
\lVert dF_p(v)\rVert_N
\geq
\lVert v\rVert_M
\]
for every $p\in M$ and every $v\in T_pM$. Then $F$ is a covering map.
\end{lem}

The preceding criterion assumes surjectivity. When the target is
connected, this assumption follows from the same metric inequality.

\begin{lem}\label{lem:covering_criterion}
Let $M$ and $N$ be connected Riemannian manifolds, with $M$ complete,
and let
\[
F:M\longrightarrow N
\]
be a local diffeomorphism. Assume that
\[
\lVert dF_p(v)\rVert_N
\geq
\lVert v\rVert_M
\]
for every $p\in M$ and every $v\in T_pM$. Then $F$ is surjective and
is a covering map onto $N$.
\end{lem}

\begin{proof}
Fix $p_0\in M$. Since $N$ is connected, for every $q\in N$ there
exists a piecewise smooth path
\[
\sigma:[0,1]\longrightarrow N,
\qquad
\sigma(0)=F(p_0),
\qquad
\sigma(1)=q.
\]
By successive continuation through local inverses of $F$, the path
$\sigma$ admits a unique maximal lift
\[
\widetilde\sigma:[0,b)\longrightarrow M,
\qquad
0<b\leq1,
\]
with $\widetilde\sigma(0)=p_0$. On each smooth segment of $\sigma$,
\[
dF_{\widetilde\sigma(t)}
\bigl(\dot{\widetilde\sigma}(t)\bigr)
=
\dot\sigma(t),
\]
and hence
\[
\lVert\dot{\widetilde\sigma}(t)\rVert_M
\leq
\lVert\dot\sigma(t)\rVert_N.
\]
It follows that, for $0\leq s<t<b$,
\[
d_M\bigl(\widetilde\sigma(s),\widetilde\sigma(t)\bigr)
\leq
\operatorname{Length}\bigl(\sigma|_{[s,t]}\bigr).
\]
Since
\[
\operatorname{Length}\bigl(\sigma|_{[s,t]}\bigr)\longrightarrow0
\qquad\text{as }s,t\longrightarrow b^-,
\]
the curve $\widetilde\sigma(t)$ is Cauchy as $t\to b^-$. By
completeness, it converges to some $p_b\in M$, and continuity gives
\[
F(p_b)=\sigma(b).
\]

If $b<1$, choose a local inverse of $F$ defined on a neighborhood of
$\sigma(b)$ and taking $\sigma(b)$ to $p_b$. Since
$\widetilde\sigma(t)\to p_b$ and $\sigma(t)\to\sigma(b)$ as
$t\to b^-$, this inverse extends $\widetilde\sigma$ beyond $b$,
contradicting maximality. Hence $b=1$. Setting
\[
\widetilde\sigma(1)=p_b,
\]
we obtain
\[
F\bigl(\widetilde\sigma(1)\bigr)=\sigma(1)=q.
\]
Thus $F$ is surjective. Lemma~\ref{lem:surjective_covering_criterion}
now shows that $F$ is a covering map.
\end{proof}

The preceding covering criterion will identify the universal cover of a
compact realization with the global parameter plane. Once this
identification is established, compactness is reduced to the structure
of the corresponding deck transformation group. We therefore record
the elementary facts about planar translation groups and finite-index
subgroups that will be used below.

\begin{lem}\label{lem:planar_quotient_groups}
\begin{enumerate}
\item Let $\Lambda\leq(\mathbb R^2,+)$ be a discrete subgroup. Then
$\mathbb R^2/\Lambda$ is compact if and only if
\[
\Lambda=\mathbb Zv_1\oplus\mathbb Zv_2
\]
for two linearly independent vectors $v_1,v_2\in\mathbb R^2$.
In particular, a discrete translation group of rank at most one does
not act cocompactly on $\mathbb R^2$.

\item Let a group $G$ act freely and properly discontinuously by
smooth diffeomorphisms on a connected manifold $Y$, and let $H\leq G$.
Then the natural map
\[
Y/H\longrightarrow Y/G
\]
is a covering map whose fibers have cardinality $[G:H]$. Hence it has
finite degree if and only if $[G:H]<\infty$, in which case its degree
is $[G:H]$. Consequently, if $Y/G$ is compact, then
\[
Y/H\text{ is compact}
\quad\Longleftrightarrow\quad
[G:H]<\infty.
\]
\end{enumerate}
\end{lem}

\begin{proof}
Let
\[
V=\operatorname{span}_{\mathbb R}\Lambda.
\]
Since $\Lambda$ is discrete, every bounded subset of $\mathbb R^2$
meets $\Lambda$ in finitely many points. If $\dim V=0$, then
$\Lambda=\{0\}$. If $\dim V=1$, choosing a shortest nonzero vector
$v_1\in\Lambda$ gives
\[
\Lambda=\mathbb Zv_1.
\]

Suppose that $\dim V=2$. Choose a shortest nonzero vector
$v_1\in\Lambda$, and let
\[
\pi:\mathbb R^2\longrightarrow v_1^\perp
\]
be the orthogonal projection. The subgroup $\pi(\Lambda)$ is discrete;
otherwise, reducing elements of $\Lambda$ modulo $\mathbb Zv_1$ would
produce infinitely many points of $\Lambda$ in a bounded set. Hence
\[
\pi(\Lambda)=\mathbb Z\pi(v_2)
\]
for some $v_2\in\Lambda\setminus\mathbb Rv_1$, and therefore
\[
\Lambda=\mathbb Zv_1\oplus\mathbb Zv_2.
\]
The parallelogram spanned by $v_1$ and $v_2$ is then a compact
fundamental domain. If $\dim V<2$, the orthogonal projection onto
$V^\perp$ descends to an unbounded continuous map on
$\mathbb R^2/\Lambda$, so the quotient is noncompact. This proves
the first assertion.

For the second assertion, freeness and proper discontinuity imply that
the natural map $Y/H\to Y/G$ is a covering. Each fiber is naturally
identified with the coset space $H\backslash G$, and hence has
cardinality $[G:H]$.

If $[G:H]<\infty$, then $Y/H$ is a finite cover of the compact manifold
$Y/G$ and is therefore compact. Conversely, if $Y/H$ is compact, each
fiber is a closed discrete subset of a compact space and hence is
finite. Thus $[G:H]<\infty$.
\end{proof}

We now combine the local-symmetry rigidity, the developing-map
construction, and Lemmas~\ref{lem:covering_criterion}
and~\ref{lem:planar_quotient_groups} to prove the compactness
criterion.

\begin{thm}
\label{thm:positive_c_compactness}
Let $0<c<1/2$ and $h\geq0$, and let $q(h,c)$ be the normalized
longitude increment defined in Definition~\ref{def:q}. Then
$\operatorname{Hel}_c^h$ is compact if and only if
\[
h\in\mathbb Q,
\qquad
q(h,c)\in\mathbb Q.
\]
\end{thm}

\begin{proof}
Write $X=X_c^h$. We first prove necessity. Let $(M,f)$ be a compact
realization of $\operatorname{Hel}_c^h$ in Definition~\ref{def:compactness}, and let
\[
p:\widetilde M\longrightarrow M
\]
be the universal covering. After shrinking the parameter charts of $M$
if necessary, assume that each chart domain $V_\alpha$ is evenly covered by $p$.
For each compatible parameter chart
\[
\Phi_\alpha:V_\alpha\longrightarrow W_\alpha\subset\mathbb R^2
\]
and each connected component $U_\alpha$ of $p^{-1}(V_\alpha)$, set
\[
D_\alpha
:=
\Phi_\alpha\circ p|_{U_\alpha}
:
U_\alpha\longrightarrow W_\alpha.
\]
Then $D_\alpha$ is a diffeomorphism and by Definition~\ref{def:compactness},
\[
f\circ p=X\circ D_\alpha
\qquad\text{on }U_\alpha.
\]

Let $C$ be a connected component of $U_\alpha\cap U_\beta$. The
transition map
\[
F_{\beta\alpha}
:=
D_\beta\circ(D_\alpha|_C)^{-1}
:
D_\alpha(C)\longrightarrow D_\beta(C)
\]
satisfies
\[
X\circ F_{\beta\alpha}=X.
\]
By Lemma~\ref{lem:local_transition_rigidity} and
\eqref{eq:global_translation_period_condition},
$F_{\beta\alpha}$ is a translation by an element of
$\operatorname{Per}(X)$. Hence the hypotheses of
Lemma~\ref{lem:translational_developing_map} are satisfied, and there
exists a local diffeomorphism
\[
D:\widetilde M\longrightarrow\mathbb R^2
\]
such that
\[
f\circ p=X\circ D.
\]

With
\[
g_M=f^*g_{\mathbb S^3},
\qquad
g=X^*g_{\mathbb S^3},
\]
we have
\[
D^*g=p^*g_M.
\]
Thus $D$ is a local isometry. Since $M$ is compact, $(M,g_M)$ is
complete. Since $p$ is a Riemannian covering,
$(\widetilde M,p^*g_M)$ is complete as well. By
Lemma~\ref{lem:covering_criterion}, $D$ is a covering of
$\mathbb R^2$. Since both spaces are simply connected, $D$ is a
diffeomorphism.

Set
\[
\pi:=p\circ D^{-1}:\mathbb R^2\longrightarrow M
\]
and let $\Gamma$ be its deck group. For every $\gamma\in\Gamma$,
\[
X\circ\gamma
=
f\circ\pi\circ\gamma
=
f\circ\pi
=
X.
\]
Lemma~\ref{lem:local_transition_rigidity} therefore gives
\[
\Gamma
\leq
\left\{
(s,t)\mapsto(s+mL,t+2\pi n):
(m,n)\in\mathbb Z^2,\ 
mq(h,c)+nh\in\mathbb Z
\right\}.
\]

The deck action is properly discontinuous, so $\Gamma$ is a discrete
subgroup of the translation group. Since
$M\cong\mathbb R^2/\Gamma$ is compact,
Lemma~\ref{lem:planar_quotient_groups} shows that $\Gamma$ is a
rank-two lattice. Hence it contains two
linearly independent period vectors
\[
(m_1L,2\pi n_1),
\qquad
(m_2L,2\pi n_2).
\]
Thus, for some $k_1,k_2\in\mathbb Z$,
\[
\begin{pmatrix}
m_1&n_1\\
m_2&n_2
\end{pmatrix}
\begin{pmatrix}
q(h,c)\\
h
\end{pmatrix}
=
\begin{pmatrix}
k_1\\
k_2
\end{pmatrix}.
\]
The coefficient matrix has nonzero determinant, so both
\[
q(h,c),h\in\mathbb Q.
\]

Conversely, suppose that both $q(h,c),h\in\mathbb Q$. Choose positive
integers $M_1,M_2$ such that
\[
M_1q(h,c)\in\mathbb Z,
\qquad
M_2h\in\mathbb Z.
\]
By \eqref{eq:global_translation_period_condition}, the vectors
\[
(M_1L,0),
\qquad
(0,2\pi M_2)
\]
are periods of $X$. Hence $X$ descends to a minimal immersion of the
compact torus
\[
\mathbb R^2/
\bigl(
\mathbb Z(M_1L,0)
\oplus
\mathbb Z(0,2\pi M_2)
\bigr)
\]
into $\mathbb S^3$, giving a compact realization of
$\operatorname{Hel}_c^h$.
\end{proof}

We next treat the Lawson case.

\begin{thm}
\label{thm:lawson_compactness}
Let $h>0$. Then $\operatorname{Hel}_0^h$ is compact if and only if
$h\in\mathbb Q$.
\end{thm}

\begin{proof}
Suppose first that $h=j/\nu\in\mathbb Q$ is written in lowest
terms. By
\eqref{eq:global_lawson_parametrization}, the vectors
\[
(2\pi,0),
\qquad
(0,2\pi\nu)
\]
are periods of $X_0^h$. Hence $X_0^h$ descends to a minimal immersion
of the compact torus
\[
\mathbb R^2/
\bigl(
\mathbb Z(2\pi,0)
\oplus
\mathbb Z(0,2\pi\nu)
\bigr)
\]
into $\mathbb S^3$. Therefore $\operatorname{Hel}_0^h$ is compact.

Conversely, suppose that $h$ is irrational and that
$\operatorname{Hel}_0^h$ is compact. Let $(M,f)$ be a compact realization, and let
\[
p:\widetilde M\longrightarrow M
\]
be the universal covering. As in the necessity part of
Theorem~\ref{thm:positive_c_compactness}, lift a compatible parameter
atlas to obtain connected open sets $U_\alpha\subset\widetilde M$ and
diffeomorphisms
\[
D_\alpha:U_\alpha\longrightarrow W_\alpha\subset\mathbb R^2
\]
such that
\[
f\circ p=X_0^h\circ D_\alpha.
\]

Let $C$ be a connected component of $U_\alpha\cap U_\beta$. The
transition map
\[
D_\beta\circ(D_\alpha|_C)^{-1}
\]
is a local symmetry of $X_0^h$. By
Corollary~\ref{cor:lawson_irrational_rigidity}, it is a translation by
a vector of the form $(2\pi r,0)$, which is a period of $X_0^h$.
Lemma~\ref{lem:translational_developing_map} therefore gives a local
diffeomorphism
\[
D:\widetilde M\longrightarrow\mathbb R^2
\]
such that
\[
f\circ p=X_0^h\circ D.
\]
The completeness and covering argument in the proof of
Theorem~\ref{thm:positive_c_compactness} shows that $D$ is a
diffeomorphism.

Identifying $\widetilde M$ with $\mathbb R^2$ through $D$, let
$\Gamma$ be the resulting deck group. Every $\gamma\in\Gamma$ satisfies
\[
X_0^h\circ\gamma=X_0^h.
\]
By Corollary~\ref{cor:lawson_irrational_rigidity}, every such
transformation has the form
\[
\gamma(s,t)=(s+2\pi r,t),
\qquad
r\in\mathbb Z.
\]
Thus $\Gamma$ is a discrete translation group of rank at most one,
which contradicts the compactness of $M\cong\mathbb R^2/\Gamma$ by
Lemma~\ref{lem:planar_quotient_groups}.
Therefore $h\in\mathbb Q$.
\end{proof}

Combining Theorems~\ref{thm:positive_c_compactness}
and~\ref{thm:lawson_compactness} yields the following complete
compactness criterion.
\begin{thm}\label{thm:compactness}
Let $c\in[0,1/2)$ and $h\geq0$. For $0<c<1/2$, let $q(h,c)$ be the
normalized longitude increment defined by \eqref{eq:q_definition}.
Then:
\begin{enumerate}
\item The parameter value $(h,c)=(0,0)$ corresponds to the compact totally
geodesic sphere.
\item If $c=0$ and $h>0$, then $\operatorname{Hel}_0^h$ is compact if and only if $h\in\mathbb Q$.
\item If $0<c<1/2$, then $\operatorname{Hel}_c^h$ is compact if and only if
$h\in\mathbb Q$ and $q(h,c)\in\mathbb Q$.
This includes the spherical catenoid subfamily $h=0$, for which the
condition reduces to $q(0,c)\in\mathbb Q$.
\end{enumerate}
\end{thm}

This proves Theorem~\ref{mainthm:compactness}.

\begin{rem}
In the rotational case $h=0$, write
\[
2c=\sin\beta,
\qquad
0<\beta<\frac{\pi}{2}.
\]
Then
\[
\mathcal K_c^0(z)=-\frac{c}{z}
=-\frac{\sin\beta}{2z},
\]
so, up to a translation of the arc-length parameter, the reconstructed
profile is the spherical catenary $\mathcal C_\beta$ of
\cite[Section~3.2]{Castro2024}. Its height function has minimal positive
period $L=\pi$, and the longitude equation agrees with
\cite[Eqs.~(3.12)--(3.14)]{Castro2024}. With the normalization
$\lambda_\beta(0)=0$ used there, it follows that
\[
q(0,c)
=
\frac{\lambda_c^0(s+\pi)-\lambda_c^0(s)}{2\pi}
=
\frac{\lambda_\beta(\pi)}{2\pi}
= T(\beta),
\]
where $T(\beta)$ is the quantity defined in
\cite[Eq.~(3.15)]{Castro2024}. Thus
\[
q(0,c)\in\mathbb Q
\quad\Longleftrightarrow\quad
T(\beta)\in\mathbb Q.
\]
Hence the rotational specialization of
Theorem~\ref{thm:compactness} is exactly the classical condition $T(\beta)\in\mathbb Q$ for a spherical catenary to be closed. By
\cite[Corollary~3.5 and Theorem~5.2]{Castro2024}, the corresponding compact
rotational minimal surfaces are the Otsuki tori. 
\end{rem}

\subsection{Range of the longitude increment}

We now determine the range of $q(h,c)$ as $c$ varies. Together with
Theorem~\ref{thm:compactness}, this yields existence results for compact
members with $0<c<1/2$. Fix $h_0\geq0$. Combining \eqref{eq:q_definition} and
\eqref{eq:longitude_drift}, and then substituting \eqref{eq:Kc},
\eqref{eq:Fhc_factorization}, and \eqref{eq:height_extrema}, gives
\begin{equation}\label{eq:q_explicit_integral}
q(h_0,c)
=
\frac{c}{\pi}
\int_{\sqrt{(1-\sqrt{1-4c^2})/2}}^{\sqrt{(1+\sqrt{1-4c^2})/2}}
\frac{\sqrt{h_0^2+(1-h_0^2)z^2}}
{z(1-z^2)\sqrt{z^2-z^4-c^2}}\,\dif z.
\end{equation}

\begin{lem}\label{lem:q_range}
Let $q(h,c)$ be the normalized longitude increment defined by
\eqref{eq:q_definition}. For every fixed $h_0\geq0$, the function
$c\mapsto q(h_0,c)$ is continuous on $(0,1/2)$. Moreover,
\begin{equation}\label{eq:q_endpoint_limits}
\lim_{c\to0^+}q(h_0,c)=\frac{1+h_0}{2},
\qquad
\lim_{c\to1/2^-}q(h_0,c)=\frac{\sqrt{2(1+h_0^2)}}{2}.
\end{equation}
If $h_0\neq1$, then the range of $q(h_0,\cdot)$ on $(0,1/2)$ is the open
interval
\begin{equation}\label{eq:q_range_interval}
\bigl\{q(h_0,c):0<c<1/2\bigr\}
=
\left(
\frac{1+h_0}{2},
\frac{\sqrt{2(1+h_0^2)}}{2}
\right).
\end{equation}
If $h_0=1$, then
\begin{equation}\label{eq:q_pitch_one}
q(1,c)\equiv1
\end{equation}
on $(0,1/2)$, and the range is the singleton $\{1\}$.
\end{lem}

\begin{proof}
We first rewrite $q(h_0,c)$ as an integral over a fixed interval. This
form gives the two endpoint limits by dominated convergence and also yields
strict pointwise bounds between the limiting values.

Set
\[
\delta=\sqrt{1-4c^2},\qquad
A_0=\frac{1+h_0^2}{2},\qquad
B_0=\frac{1-h_0^2}{2}.
\]
Put $t=z^2$, and then from $z^2-z^4-c^2\geq0$ in \eqref{eq:q_explicit_integral} we change the variables by
\[
t=\frac{1-\delta\cos\theta}{2},\qquad 0\le\theta\le\pi.
\]
The elementary identities needed for this change of variables are
\[
    t(1-t)=\frac{1-\delta^2\cos^2\theta}{4},
    \qquad
    t-t^2-c^2=\frac{\delta^2\sin^2\theta}{4},
    \qquad
    \dif t=\frac{\delta}{2}\sin\theta\,\dif\theta.
\]
Moreover
\[
    h_0^2+(1-h_0^2)t=A_0-B_0\delta\cos\theta.
\]
Since
\[
    \frac{\dif z}{z(1-z^2)\sqrt{z^2-z^4-c^2}}
    =
    \frac{\dif t}{2t(1-t)\sqrt{t-t^2-c^2}},
\]
the integral \eqref{eq:q_explicit_integral}  becomes
\[
q(h_0,c)=
\frac{2c}{\pi}
\int_0^\pi
\frac{\sqrt{A_0-B_0\delta\cos\theta}}
{1-\delta^2\cos^2\theta}
\,\dif\theta.
\]
Using the symmetry $\cos(\pi-\theta)=-\cos\theta$, this is equivalently
\begin{equation}\label{eq:q_std}
q(h_0,c)
=
\frac{2c}{\pi}
\int_0^{\pi/2}
\frac{
\sqrt{A_0-B_0\delta\cos\theta}
+
\sqrt{A_0+B_0\delta\cos\theta}
}
{1-\delta^2\cos^2\theta}
\,\dif\theta.
\end{equation}
On every compact subinterval of $(0,1/2)$, the denominator in
\eqref{eq:q_std} is bounded away from zero and the integrand depends
continuously on $c$. Hence $q(h_0,c)$ is continuous on $(0,1/2)$.

Direct calculations show that
\begin{equation}\label{eq:int_aux}
\int_0^{\pi/2}
\frac{\dif\theta}{1-\delta^2\cos^2\theta}
=
\frac{\pi}{4c}.
\end{equation}

For the limit $c\to0^+$, write $u=\tan\theta$ in
\eqref{eq:q_std} and then set $u=2cv$. This gives
\[
q(h_0,c)
=
\frac1\pi
\int_0^\infty
\frac{
\Sigma(\delta,\arctan(2cv))
}
{1+v^2}\,\dif v,
\]
where
\[
\Sigma(\delta,\theta)
=
\sqrt{A_0-B_0\delta\cos\theta}
+
\sqrt{A_0+B_0\delta\cos\theta}.
\]
As $c\to0^+$, $\delta\to1$ and
\[
\Sigma(\delta,\arctan(2cv))\to \Sigma(1,0)=1+h_0.
\]
Since $\Sigma$ is uniformly bounded, dominated convergence gives
\[
\lim_{c\to0^+}q(h_0,c)
=
\frac1\pi
\int_0^\infty
\frac{1+h_0}{1+v^2}\,\dif v
=
\frac{1+h_0}{2}.
\]

For the limit $c\to1/2^-$, we have $\delta\to0$ and $2c\to1$.
For $\delta\leq1/2$, the denominator in \eqref{eq:q_std} is bounded
below by $3/4$, while the numerator is uniformly bounded. Hence dominated
convergence gives
\[
\lim_{c\to1/2^-}q(h_0,c)
=
\frac1\pi
\int_0^{\pi/2}2\sqrt{A_0}\,\dif\theta
=
\sqrt{A_0}
=
\frac{\sqrt{2(1+h_0^2)}}{2}.
\]

If $h_0=1$, then $A_0=1$ and $B_0=0$, so $\Sigma(\delta,\theta)=2$. Using
\eqref{eq:q_std} and \eqref{eq:int_aux}, we get
\[
q(1,c)=1
\]
for every $0<c<1/2$.

Now assume $h_0\neq1$. Then $B_0\neq0$. We verify the strict pointwise
inequality separately because the sign of $B_0$ changes at $h_0=1$.  For
$0\le y\le |B_0|$, set
\[
\mathcal R(y)=\sqrt{A_0-y}+\sqrt{A_0+y}.
\]
Then
\[
\mathcal R'(y)=\frac{1}{2\sqrt{A_0+y}}-\frac{1}{2\sqrt{A_0-y}}<0
\qquad(0<y<|B_0|),
\]
and
\[
\mathcal R(|B_0|)=1+h_0,
\qquad
\mathcal R(0)=2\sqrt{A_0}.
\]
Thus, whenever $0<|B_0|x<|B_0|$,
\[
1+h_0<\mathcal R(|B_0|x)<2\sqrt{A_0}.
\]
Since
\[
\sqrt{A_0-B_0x}+\sqrt{A_0+B_0x}=\mathcal R(|B_0|x),
\]
this applies to $x=\delta\cos\theta$.  For $0<c<1/2$, we have $0<\delta<1$, and thus the strict
inequalities above hold for all $\theta\in[0,\pi/2]$ except for
$\theta=\pi/2$, which is irrelevant for the integral. Combining this with
\eqref{eq:q_std} and
\eqref{eq:int_aux}, we obtain the strict bounds
\[
\frac{1+h_0}{2}
<
q(h_0,c)
<
\sqrt{A_0}
=
\frac{\sqrt{2(1+h_0^2)}}{2},\quad \text{for}\quad 0<c<1/2.
\]
Since $q(h_0,\cdot)$ is continuous, its image is an interval.  The endpoint limits in \eqref{eq:q_endpoint_limits} give its infimum and
supremum, while the strict inequalities exclude the endpoints. Therefore, for $h_0\neq1$,
\[
q(h_0,(0,1/2))
=
\left(
\frac{1+h_0}{2},
\frac{\sqrt{2(1+h_0^2)}}{2}
\right).
\]
This proves the lemma.
\end{proof}

\begin{coro}\label{cor:infinitely_many_compact}
The two-parameter family
\[
\{\operatorname{Hel}_c^h:h\geq0,\ 0\leq c<1/2\}
\]
contains infinitely many compact members with $0<c<1/2$.
More precisely, for every
\[
h_0\in\mathbb Q\cap[0,\infty),
\qquad
h_0\neq1,
\]
there exist infinitely many $c\in(0,1/2)$ such that
$\operatorname{Hel}_c^{h_0}$ is compact. For $h_0=1$, every
$c\in(0,1/2)$ gives a compact member; the geometric
interpretation of the family with $h=1$ is discussed in
Remark~\ref{rem:exceptional_pitch_h_one}.
\end{coro}

\begin{proof}
Fix $h_0\in\mathbb Q\cap[0,\infty)$. We use the two cases in
Lemma~\ref{lem:q_range}. If $h_0=1$, then
\eqref{eq:q_pitch_one} gives $q(1,c)=1$ for every $c\in(0,1/2)$. Hence every $\operatorname{Hel}_c^1$ with $0<c<1/2$ is compact by
Theorem~\ref{thm:positive_c_compactness}.

Now assume $h_0\neq1$. By \eqref{eq:q_range_interval}, the image of
$c\mapsto q(h_0,c)$ on $(0,1/2)$ is the open interval
$I$ given by the right-hand side of \eqref{eq:q_range_interval}.
Since $\mathbb Q$ is dense in $\mathbb R$, the interval $I$ contains
infinitely many distinct rational numbers $r_1,r_2,\ldots$. For each $r_j$,
there exists $c_j\in(0,1/2)$ such that
\[
q(h_0,c_j)=r_j.
\]
Since $h_0\in\mathbb Q$ and $q(h_0,c_j)\in\mathbb Q$,
Theorem~\ref{thm:positive_c_compactness} implies that
$\operatorname{Hel}_{c_j}^{h_0}$ is compact. The values $r_j$ are distinct, so the corresponding values $c_j$ are
distinct.

Thus the family contains infinitely many compact members with $0<c<1/2$.
\end{proof}

\begin{rem}\label{rem:exceptional_pitch_h_one}
The one-parameter family
\[
\{\operatorname{Hel}_c^1:0<c<1/2\}
\]
appearing in Corollary~\ref{cor:infinitely_many_compact} does not consist
of geometrically distinct surfaces. By
\cite[Remark~5.8]{Castro2024}, the immersions with $h=1$ are Hopf minimal
surfaces and give helicoidal parametrizations of the Clifford torus.
Thus varying $c$ changes the helicoidal parametrization, but not the
congruence class of the underlying surface.
\end{rem}

\section{Automorphisms and Compact Quotients}
\label{sec:compact_quotients}

For a global parametrization
$X:\mathbb R^2\to\mathbb S^3$, define
\begin{equation}\label{eq:automorphism_group_definition}
\operatorname{Aut}(X)
:=
\left\{
\Phi\in\operatorname{Diff}(\mathbb R^2):X\circ\Phi=X
\right\}.
\end{equation}
We call the group defined in
\eqref{eq:automorphism_group_definition} the automorphism group of the
parametrization $X$.
For $v\in\mathbb R^2$, write $\tau_v(u)=u+v$. The period group
$\operatorname{Per}(X)$ is defined in
\eqref{eq:translation_period_definition_general}; the translations
$\tau_v$ with $v\in\operatorname{Per}(X)$ form a subgroup of
$\operatorname{Aut}(X)$. We now determine whether additional,
nontranslational automorphisms occur and classify the resulting compact
quotient surfaces in terms of finite-index subgroups.

\subsection{Automorphisms and quotients for \texorpdfstring{$0<c<1/2$}{0<c<1/2}}

\begin{coro}
\label{cor:positive_c_symmetry_group}
Let $h\geq0$ and $0<c<1/2$. Let $L$ be the minimal positive period of $z_c^h$ given
by Lemma~\ref{lem:minimal_height_period}, let $q(h,c)$ be defined by
\eqref{eq:q_definition}, and let $\operatorname{Per}(X_c^h)$ be defined
by \eqref{eq:translation_period_definition_general}. Then
\begin{equation}\label{eq:positive_c_period_set}
\operatorname{Per}(X_c^h)
=
\left\{
(mL,2\pi n):
 m,n\in\mathbb Z,\;
 mq(h,c)+nh\in\mathbb Z
\right\}.
\end{equation}
Every $C^1$-diffeomorphism
\[
F:\Omega\longrightarrow\Omega'
\]
between nonempty connected open subsets of $\mathbb R^2$ satisfying
$X_c^h\circ F=X_c^h$ is the restriction of a translation
\[
\tau_v(u)=u+v,
\qquad
v\in\operatorname{Per}(X_c^h).
\]
Consequently, every automorphism of the parametrization is a translation, and
\begin{equation}\label{eq:positive_c_automorphism_group}
\operatorname{Aut}(X_c^h)
=
\left\{\tau_v:v\in\operatorname{Per}(X_c^h)\right\}.
\end{equation}
\end{coro}

\begin{proof}
By Lemma~\ref{lem:local_transition_rigidity}, every such $F$ has the form
$F(s,t)=(s+mL,t+2\pi n)$ on $\Omega$, for some
$m,n\in\mathbb Z$ satisfying $mq(h,c)+nh\in\mathbb Z$. Thus $F$
is the restriction of the translation
$\tau_v$, where $v=(mL,2\pi n)\in\operatorname{Per}(X_c^h)$.
Taking $\Omega=\Omega'=\mathbb R^2$ in the same lemma gives the global
assertion.
\end{proof}

Throughout this subsection, assume that $0<c<1/2$. Let $L$ be the
minimal positive period of the height function from
Lemma~\ref{lem:minimal_height_period}, put $q=q(h,c)$, and define
\begin{equation}\label{eq:translation_period_set}
\Lambda_X:=\operatorname{Per}(X_c^h),
\qquad
\Lambda_{\mathrm{par}}:=\langle(L,0),(0,2\pi)\rangle_{\mathbb Z}.
\end{equation}
Thus $\Lambda_X$ is the period group of $X_c^h$, while
$\Lambda_{\mathrm{par}}$ is the rectangular lattice generated by the vectors
$(L,0)$ and $(0,2\pi)$; here
$\langle\cdot\rangle_{\mathbb Z}$ denotes the integer span.
By Corollary~\ref{cor:positive_c_symmetry_group}, equivalently
\eqref{eq:positive_c_period_set}, $\Lambda_X$ consists exactly of the
period vectors in $\Lambda_{\mathrm{par}}$ satisfying the arithmetic
condition stated there.

When $\operatorname{Hel}_c^h$ is compact,
Theorem~\ref{thm:compactness} gives $q,h\in\mathbb Q$. Moreover,
$q=q(h,c)>0$ by Lemma~\ref{lem:q_range} (also by \eqref{eq:lambda_derivative} and
\eqref{eq:q_definition}). Write
\begin{equation}\label{eq:reduced_closing_data}
q=\frac{p}{r},
\qquad
h=\frac{j}{\nu}
\end{equation}
in lowest terms, with $p,r>0$, $j\geq0$, $\nu>0$, and
$\nu=1$ if $h=0$.

The following index is the multiplicity relating the
parameter-rectangle integral to the Willmore energy of the quotient in
Section~\ref{sec:willmore_energy}.

\begin{lem}\label{lem:period_lattice_index}
Let $h\geq0$ and $0<c<1/2$, assume that $q=q(h,c),h\in\mathbb Q$, and
write $q=p/r$ and $h=j/\nu$ as in
\eqref{eq:reduced_closing_data}. Then $\Lambda_X$ is a rank-two lattice
and
\begin{equation}\label{eq:translation_lattice_index}
[\Lambda_{\mathrm{par}}:\Lambda_X]
=
\operatorname{lcm}(r,\nu),
\end{equation}
where $\Lambda_X$ and $\Lambda_{\mathrm{par}}$ are defined in
\eqref{eq:translation_period_set}.
\end{lem}

\begin{proof}
Identify $\Lambda_{\mathrm{par}}$ with $\mathbb Z^2$ via
\[
(m,n)\longleftrightarrow(mL,2\pi n).
\]
By \eqref{eq:positive_c_period_set}, the subgroup $\Lambda_X$
corresponds to the kernel of the homomorphism
\[
\chi:\mathbb Z^2\longrightarrow\mathbb R/\mathbb Z,
\qquad
\chi(m,n)=mq+nh+\mathbb Z.
\]

 Since
\[
\gcd(p,r)=1,
\qquad
\gcd(j,\nu)=1,
\]
we have
\[
\operatorname{im}\chi
=
\left\langle
\frac{1}{r}+\mathbb Z,
\frac{1}{\nu}+\mathbb Z
\right\rangle_{\mathbb Z}\subset \mathbb R/\mathbb Z.
\]
Set
\[
N=\operatorname{lcm}(r,\nu).
\]
Then
\[
\frac{1}{r}+\mathbb Z,
\frac{1}{\nu}+\mathbb Z
\in
\left\langle
\frac{1}{N}+\mathbb Z
\right\rangle_{\mathbb Z}.
\]
Conversely,
\[
\gcd\left(\frac{N}{r},\frac{N}{\nu}\right)=1,
\]
so B\'ezout's identity shows that
\[
\frac{1}{N}+\mathbb Z
\in
\left\langle
\frac{1}{r}+\mathbb Z,
\frac{1}{\nu}+\mathbb Z
\right\rangle_{\mathbb Z}.
\]
Hence
\[
\operatorname{im}\chi
=
\left\langle
\frac{1}{N}+\mathbb Z
\right\rangle_{\mathbb Z}\subset \mathbb R/\mathbb Z,
\]
which has order $N$. Therefore
\[
|\operatorname{im}\chi|
=
\operatorname{lcm}(r,\nu).
\]

By the first isomorphism theorem,
\[
[\Lambda_{\mathrm{par}}:\Lambda_X]
=
[\mathbb Z^2:\ker\chi]
=
|\operatorname{im}\chi|
=
\operatorname{lcm}(r,\nu),
\]
which proves \eqref{eq:translation_lattice_index}. In particular,
$\Lambda_X$ has finite index in the rank-two lattice
$\Lambda_{\mathrm{par}}$, and hence is itself a rank-two lattice.
\end{proof}

Assume now that $\operatorname{Hel}_c^h$ is compact. By
\eqref{eq:positive_c_automorphism_group} and
Lemma~\ref{lem:period_lattice_index}, $\operatorname{Aut}(X_c^h)$ is the
group of translations associated with the rank-two lattice $\Lambda_X$.
Hence it acts
freely, properly discontinuously, and cocompactly on $\mathbb R^2$. We denote the resulting quotient torus by
\begin{equation}\label{eq:positive_c_automorphism_quotient}
Q_c^h
:=
\mathbb R^2/\operatorname{Aut}(X_c^h)
=
\mathbb R^2/\Lambda_X.
\end{equation}
\begin{thm}
\label{thm:positive_c_compact_realization_classification}
Let $h\geq0$ and $0<c<1/2$, and assume that
$\operatorname{Hel}_c^h$ is compact. Let
$\Lambda_X$ be the period lattice defined in
\eqref{eq:translation_period_set}. For every compact realization $(M,f)$ of
$\operatorname{Hel}_c^h$ in the sense of
Definition~\ref{def:compactness}, there exists a finite-index
sublattice $\Gamma\leq\Lambda_X$ such that
$M\cong\mathbb R^2/\Gamma$ is a torus
and, under this identification, $f$ is induced by $X_c^h$.
Conversely, every finite-index sublattice
$\Gamma\leq\Lambda_X$ determines a compact realization of
$\operatorname{Hel}_c^h$.
\end{thm}

\begin{proof}
Let $(M,f)$ be a compact realization of
$\operatorname{Hel}_c^h$. The proof of
Theorem~\ref{thm:positive_c_compactness} identifies the universal
cover of $M$ with $\mathbb R^2$. Under this identification, its deck
group consists of translations
\[
\{\tau_v:v\in\Gamma\},
\]
where $\Gamma\leq\Lambda_X$ is a rank-two sublattice. Since both $\mathbb R^2/\Gamma\cong M$ and
$\mathbb R^2/\Lambda_X=Q_c^h$ are compact,
Lemma~\ref{lem:planar_quotient_groups} gives
$[\Lambda_X:\Gamma]<\infty$.
It follows that
$M\cong\mathbb R^2/\Gamma$ is a torus
and $f$ is induced by $X_c^h$.

Conversely, let $\Gamma\leq\Lambda_X$ be a finite-index sublattice.
Then $X_c^h$ is $\Gamma$-invariant and descends to an immersion
\[
\mathbb R^2/\Gamma\longrightarrow\mathbb S^3.
\]
By Lemma~\ref{lem:planar_quotient_groups},
$\mathbb R^2/\Gamma$ is compact, so the resulting pair is a compact
realization of $\operatorname{Hel}_c^h$.
\end{proof}

\subsection{Lawson automorphisms and quotients}

In the Lawson case $c=0$, the automorphism group may contain glide reflections
in addition to translations.

For $h>0$, write
\begin{equation}\label{eq:lawson_automorphism_group}
\mathcal G_0^h:=\operatorname{Aut}(X_0^h)
\end{equation}
for the automorphism group of the Lawson parametrization. Its local form is
given by Lemma~\ref{lem:lawson_complete_local_symmetry}.

\begin{prop}
\label{prop:lawson_complete_symmetry_group}
Let $h=j/\nu>0$, where $j$ and $\nu$ are coprime positive integers. The period lattice of $X_0^h$ is
\begin{equation}\label{eq:lawson_translation_lattice}
\Lambda_0^h
:=
\operatorname{Per}(X_0^h)
=
\begin{cases}
\langle(\pi,\pi\nu),(0,2\pi\nu)\rangle_{\mathbb Z},
& j,\nu \text{ are both odd},\\[1mm]
\langle(2\pi,0),(0,2\pi\nu)\rangle_{\mathbb Z},
& \text{exactly one of $j,\nu$ is even},
\end{cases}.
\end{equation}
Let
\begin{equation}\label{eq:lawson_translation_subgroup}
\mathcal T_0^h
:=
\{\tau_v:v\in\Lambda_0^h\}
\end{equation}
be the corresponding translation subgroup. Then $\mathcal T_0^h$ is the
orientation-preserving subgroup of
$\mathcal G_0^h$. More precisely, a translation
\[
(s,t)\longmapsto(s+a,t+b)
\]
belongs to $\mathcal G_0^h$ if and only if
\begin{equation}\label{eq:lawson_translation_period_conditions}
(a,b)=(k\pi,m\pi),
\qquad
m+k\in2\mathbb Z,
\qquad
hm+k\in2\mathbb Z
\end{equation}
for some $k,m\in\mathbb Z$.

If $j$ and $\nu$ are both odd, then
\begin{equation}\label{eq:lawson_group_odd_odd}
\mathcal G_0^h=\mathcal T_0^h,
\end{equation}
and $\mathbb R^2/\mathcal G_0^h$ is a torus.

If exactly one of $j$ and $\nu$ is even, define
\begin{equation}\label{eq:lawson_glide_reflection}
\rho_{j,\nu}(s,t)=(-s-j\pi,t+\nu\pi).
\end{equation}
Then
\begin{equation}\label{eq:lawson_group_mixed_parity}
\mathcal G_0^h
=
\mathcal T_0^h
\sqcup
\mathcal T_0^h\rho_{j,\nu},
\qquad
[\mathcal G_0^h:\mathcal T_0^h]=2.
\end{equation}
In this case $\mathbb R^2/\mathcal G_0^h$ is a Klein bottle, and
$\mathbb R^2/\Lambda_0^h$ is its orientable double
cover.
\end{prop}

\begin{proof}
By Lemma~\ref{lem:lawson_complete_local_symmetry}, the
orientation-preserving elements are precisely the translations satisfying
\eqref{eq:lawson_translation_period_conditions}. For
$\varepsilon\in\{1,-1\}$, define
\[
A_\varepsilon
:=
\left\{
(k,m)\in\mathbb Z^2:
 m+k\in2\mathbb Z+\frac{1-\varepsilon}{2},
\quad
jm+k\nu\in2\nu\mathbb Z
\right\}.
\]
Thus $A_1$ parametrizes the translations and $A_{-1}$ parametrizes the
orientation-reversing symmetries. By
Lemma~\ref{lem:lawson_complete_local_symmetry},
\[
\mathcal G_0^h
=
\left\{
(s,t)\longmapsto(\varepsilon s+k\pi,t+m\pi):
\varepsilon\in\{1,-1\},\ (k,m)\in A_\varepsilon
\right\}.
\]

Suppose first that $j$ and $\nu$ are both odd. Then $j$ is invertible
modulo $2\nu$, and its inverse is odd; hence
$j^{-1}\nu\equiv\nu\pmod{2\nu}$. Therefore
\[
jm+k\nu\in2\nu\mathbb Z
\quad\Longleftrightarrow\quad
m+k\nu\in2\nu\mathbb Z.
\]
Consequently,
\[
A_1
=
\mathbb Z(1,\nu)\oplus\mathbb Z(0,2\nu),
\qquad
A_{-1}=\varnothing.
\]
Indeed, if $(k,m)\in A_{-1}$, then $m+k$ is odd, whereas
\[
jm+k\nu\equiv m+k\pmod2,
\]
contradicting $jm+k\nu\in2\nu\mathbb Z$. Thus
\eqref{eq:lawson_translation_lattice} and
\eqref{eq:lawson_group_odd_odd} follow in this case.

Suppose now that exactly one of $j$ and $\nu$ is even. Let
$(k,m)\in A_1$. If $j$ is even and $\nu$ is odd, reduction of
$jm+k\nu\in2\nu\mathbb Z$ modulo $2$ shows that $k$ is even. If $j$
is odd and $\nu$ is even, the same divisibility condition implies
$\nu\mid jm$; since $\gcd(j,\nu)=1$, one has $\nu\mid m$, and hence
$m$ is even. Since $m+k\in2\mathbb Z$, both $k$ and $m$ are even in
either case. Writing $k=2a$ and $m=2b$ gives
\[
jb+a\nu\in\nu\mathbb Z.
\]
Since $\gcd(j,\nu)=1$, one has $b\in\nu\mathbb Z$, and hence
\[
A_1
=
\mathbb Z(2,0)\oplus\mathbb Z(0,2\nu).
\]
Moreover,
\[
A_{-1}=(-j,\nu)+A_1.
\]
Indeed, since $j+\nu$ is odd,
\[
(k,m)\in A_{-1}
\quad\Longleftrightarrow\quad
(k+j,m-\nu)\in A_1,
\]
because
\[
j(m-\nu)+(k+j)\nu=jm+k\nu.
\]
The element corresponding to $(\varepsilon,k,m)=(-1,-j,\nu)$ is
$\rho_{j,\nu}$, so
\eqref{eq:lawson_group_mixed_parity} follows.

Every nonidentity translation is fixed-point-free. If an
orientation-reversing element
\[
(s,t)\longmapsto(-s+k\pi,t+m\pi)
\]
had a fixed point, then $m=0$. Its defining conditions would then give
both $k\in2\mathbb Z+1$ and $k=hm+k\in2\mathbb Z$, a contradiction.
Hence the action of $\mathcal G_0^h$ is free. The subgroup $\mathcal T_0^h$ is the group of translations associated
with the rank-two lattice $\Lambda_0^h$, so it acts properly discontinuously
and cocompactly. Since $\mathcal G_0^h$ is the union of at most two cosets
of $\mathcal T_0^h$, the same holds for $\mathcal G_0^h$.

When $j$ and $\nu$ are both odd, the quotient is the torus determined
by $\Lambda_0^h$. In the remaining case, set
$\tau_1:=\tau_{(2\pi,0)}$. Then
\[
\mathcal T_0^h
=
\langle \tau_1,\rho_{j,\nu}^{\,2}\rangle,
\qquad
\rho_{j,\nu}^{\,2}=\tau_{(0,2\pi\nu)},
\qquad
\rho_{j,\nu}\tau_1\rho_{j,\nu}^{-1}=\tau_1^{-1}.
\]
Thus $\mathcal G_0^h=\langle \tau_1,\rho_{j,\nu}\rangle$ is generated
by a translation and a glide reflection satisfying the standard Klein bottle
relation. Hence
$\mathbb R^2/\mathcal G_0^h$ is a Klein bottle, and the quotient by its index-two translation subgroup,
$\mathbb R^2/\Lambda_0^h$, is the orientable double cover.
\end{proof}

This recovers the classical Lawson classification into tori and Klein
bottles; see \cite[Section~7]{Lawson1970} and the explicit formulation in
\cite[Definition~2]{Penskoi2012}.

For $h=j/\nu>0$ in lowest terms, Proposition~\ref{prop:lawson_complete_symmetry_group}
shows that $\mathcal G_0^h=\operatorname{Aut}(X_0^h)$ acts freely,
properly discontinuously, and cocompactly on $\mathbb R^2$. We denote the
resulting quotient surface by
\begin{equation}\label{eq:lawson_automorphism_quotient}
Q_0^h
:=
\mathbb R^2/\operatorname{Aut}(X_0^h)
=
\mathbb R^2/\mathcal G_0^h.
\end{equation}
\begin{thm}
\label{thm:lawson_compact_realization_classification}
Let $h=j/\nu>0$ be written in lowest terms, and let
$\mathcal G_0^h$ be the automorphism group defined in
\eqref{eq:lawson_automorphism_group} and described in
Proposition~\ref{prop:lawson_complete_symmetry_group}. For every compact
realization $(M,f)$ of $\operatorname{Hel}_0^h$ in the sense of
Definition~\ref{def:compactness}, there exists a finite-index subgroup
$\Gamma\leq\mathcal G_0^h$ such that
\[
M\cong\mathbb R^2/\Gamma,
\]
and, under this identification, $f$ is induced by $X_0^h$.
Conversely, every finite-index subgroup
$\Gamma\leq\mathcal G_0^h$ determines a compact realization of
$\operatorname{Hel}_0^h$.
\end{thm}

\begin{proof}
Let $(M,f)$ be a compact realization of
$\operatorname{Hel}_0^h$, and let
\[
p:\widetilde M\longrightarrow M
\]
be the universal covering. After shrinking the parameter charts
induced by $X_0^h$ if necessary, assume that each chart domain is evenly
covered by $p$. For each refined parameter chart
\[
\Phi_\alpha:V_\alpha\longrightarrow W_\alpha\subset\mathbb R^2
\]
and each connected component $U_\alpha$ of $p^{-1}(V_\alpha)$, set
\[
D_\alpha
:=
\Phi_\alpha\circ p|_{U_\alpha}
:
U_\alpha\longrightarrow W_\alpha.
\]
Then $D_\alpha$ is a diffeomorphism and
$f\circ p=X_0^h\circ D_\alpha$ on $U_\alpha$. Since the simply
connected surface $\widetilde M$ is orientable, fix an orientation on
$\widetilde M$.

Suppose first that exactly one of $j$ and $\nu$ is even. Let
$\rho_{j,\nu}$ be the glide reflection defined in
\eqref{eq:lawson_glide_reflection}. For each lifted parameter chart, set
\[
\widehat D_\alpha=
\begin{cases}
D_\alpha,
& D_\alpha\text{ is orientation preserving},\\
\rho_{j,\nu}\circ D_\alpha,
& D_\alpha\text{ is orientation reversing}.
\end{cases}
\]
Since $X_0^h\circ\rho_{j,\nu}=X_0^h$, these charts satisfy
\[
f\circ p=X_0^h\circ\widehat D_\alpha.
\]
On each connected component of an overlap, the transition map between the
charts $\widehat D_\alpha$ is an orientation-preserving local symmetry of
$X_0^h$, and hence is a translation by a vector in $\Lambda_0^h$.

If $j$ and $\nu$ are both odd, an orientation-reversing transition
would, by Lemma~\ref{lem:lawson_complete_local_symmetry}, have integers
$k,m$ with $m+k$ odd and
\[
jm+k\nu\in2\nu\mathbb Z.
\]
Reducing the latter condition modulo $2$ gives
$m+k\equiv0\pmod2$, a contradiction. Hence, on each connected
component of an overlap, the transition map is a translation by a vector in
$\Lambda_0^h$.

Lemma~\ref{lem:translational_developing_map} therefore gives a local
diffeomorphism
\[
D:\widetilde M\longrightarrow\mathbb R^2,
\qquad
f\circ p=X_0^h\circ D.
\]
As in the proof of Theorem~\ref{thm:lawson_compactness}, the pullback
metrics make $D$ a local isometry from a complete manifold. Hence
Lemma~\ref{lem:covering_criterion} shows that $D$ is a covering of
$\mathbb R^2$, and therefore a diffeomorphism.

Conjugating the deck group of $p$ by $D$ gives a subgroup
$\Gamma\leq\mathcal G_0^h$ such that
\[
M\cong\mathbb R^2/\Gamma,
\]
and $f$ is induced by $X_0^h$. Since both $\mathbb R^2/\Gamma\cong M$ and
$\mathbb R^2/\mathcal G_0^h=Q_0^h$ are compact,
Lemma~\ref{lem:planar_quotient_groups} gives
$[\mathcal G_0^h:\Gamma]<\infty$.

Conversely, let $\Gamma\leq\mathcal G_0^h$ be a finite-index
subgroup. Proposition~\ref{prop:lawson_complete_symmetry_group} shows that
$\mathcal G_0^h$ acts freely, properly discontinuously, and cocompactly on
$\mathbb R^2$. Freeness and proper discontinuity pass to $\Gamma$, while
Lemma~\ref{lem:planar_quotient_groups} gives cocompactness from
$[\mathcal G_0^h:\Gamma]<\infty$. Since $X_0^h$ is $\Gamma$-invariant, it
descends to an immersion
\[
\mathbb R^2/\Gamma\longrightarrow\mathbb S^3.
\]
The resulting pair is a compact realization of
$\operatorname{Hel}_0^h$.
\end{proof}

\section{Willmore Energy}
\label{sec:willmore_energy}

We first recall the Willmore functional, then compute the area integral over
a standard parameter rectangle and evaluate it in terms of complete elliptic
integrals. We next compute the Willmore energies of the quotient surfaces by the
automorphism groups identified in
Section~\ref{sec:compact_quotients}, and conclude with elliptic-integral
estimates and monotonicity properties. The resulting energy bounds are used
in Section~\ref{sec:associated_families}.

For a smooth immersion $f:M\to\mathbb S^3$ of a compact surface,
choose a unit normal locally and let $H_f$ denote the corresponding scalar
mean curvature, with the convention used in \cite{Castro2024}, so that
$2H_f$ is the trace of the second fundamental form with respect to the
induced metric. Since changing the local unit normal changes only the sign of
$H_f$, the function $H_f^2$ is globally defined.
We set
\begin{equation}\label{eq:willmore_functional}
    \mathcal W(f)=\int_M\bigl(1+H_f^2\bigr)\,\dif A_f,
\end{equation}
where $\dif A_f$ is the induced area measure. When a parametrization descends to the quotient surface $Q$ by its
automorphism group, let
\[
\bar X:Q\longrightarrow\mathbb S^3
\]
denote the induced immersion, and write
\begin{equation}\label{eq:quotient_willmore_energy}
\mathcal W(Q):=\mathcal W(\bar X).
\end{equation}
Since the quotient immersions considered here are minimal, their Willmore
energy equals their area.

\subsection{Area integral over a parameter rectangle}

We compute the area density of $X_c^h$ and integrate it over one height
period and one angular period. For $0<c<1/2$, write
\[
\alpha(z):=\sqrt{\det(g_{ij})}
\]
for the area density of $X_c^h$. By \eqref{eq:Kc} and the induced metric
formula \eqref{eq:local_induced_metric},
\begin{equation}\label{eq:alpha2_final}
\alpha(z)
=
\frac{z^2\sqrt{h^2+(1-h^2)z^2}}{\sqrt{z^4+h^2c^2}},
\end{equation}
where $z=z_c^h(s)$ is the global height function introduced in \eqref{eq:z_periodic_global}.
On each monotone branch of $z$, the factorization
\eqref{eq:Fhc_factorization} and \eqref{eq:Castro_reconstruction} gives
\begin{equation}\label{eq:ds_final}
\dif s
=
\frac{\sqrt{z^4+h^2c^2}}
{z\sqrt{z^2-z^4-c^2}}\,|\dif z|.
\end{equation}

Let $L$ be the minimal positive period given by
Lemma~\ref{lem:minimal_height_period}. For $0<c<1/2$, the area density is
invariant under translations by vectors $(L,0)$ and $(0,2\pi)$; for $c=0$, it is
invariant under translations by vectors $(\pi,0)$ and $(0,2\pi)$. These rectangles
need not be fundamental domains for the automorphism groups.

For $0<c<1/2$ and any $s_0\in\mathbb R$, define
\begin{equation}\label{eq:parameter_rectangle_integral}
\mathcal E(h,c)
:=
\int_0^{2\pi}\int_{s_0}^{s_0+L}\alpha(z(s))\,\dif s\,\dif t.
\end{equation}
This value is independent of $s_0$. For $c=0$ and $h\geq0$, define
\begin{equation}\label{eq:lawson_parameter_rectangle_integral}
\mathcal E(h,0)
:=
\int_0^{2\pi}\int_0^\pi
\sqrt{h^2\cos^2s+\sin^2s}\,\dif s\,\dif t.
\end{equation}
In particular,
\begin{equation}\label{eq:sphere_parameter_integral}
\mathcal E(0,0)=4\pi.
\end{equation}
At $(h,c)=(0,0)$, this is the parameter integral associated with the
polar parametrization and equals the Willmore energy of the totally
geodesic sphere.

The increasing and decreasing branches over one height period contribute
equally. Hence \eqref{eq:alpha2_final} and \eqref{eq:ds_final} yield that for $0<c<1/2$, \eqref{eq:parameter_rectangle_integral} can be written as
\begin{equation}\label{eq:parameter_integral_z}
\mathcal E(h,c)
=
4\pi\int_{z_{\min}}^{z_{\max}}
\frac{z\sqrt{h^2+(1-h^2)z^2}}
{\sqrt{z^2-z^4-c^2}}\,\dif z,
\end{equation}
where $[z_{\min},z_{\max}]$ is the range of $z$ by \eqref{eq:height_extrema} and \eqref{eq:z_periodic_global}. For $c=0$, \eqref{eq:parameter_integral_z} reduces to \eqref{eq:lawson_parameter_rectangle_integral}. 
The integral is finite because its only endpoint singularities are of
inverse-square-root type.

\subsection{Elliptic-integral evaluation}

Set $\delta=\sqrt{1-4c^2}$. Then
$z_{\min}^2=(1-\delta)/2$ and $z_{\max}^2=(1+\delta)/2$. Under the
substitution
\[
z^2=\frac{1-\delta}{2}+\delta\sin^2\theta,
\qquad 0\leq\theta\leq\frac{\pi}{2},
\]
one has
$\sqrt{z^2-z^4-c^2}=\delta\sin\theta\cos\theta$ and
$z\,\dif z=\delta\sin\theta\cos\theta\,\dif\theta$. Define
\begin{equation}\label{eq:AB}
A(h,\delta)
:=
\frac{1+\delta+(1-\delta)h^2}{2},
\qquad
B(h,\delta):=\delta(1-h^2).
\end{equation}
After replacing $\theta$ by $\pi/2-\theta$,
\eqref{eq:parameter_integral_z} becomes
\begin{equation}\label{eq:parameter_integral_theta}
\mathcal E(h,c)
=
4\pi\int_0^{\pi/2}
\sqrt{A(h,\delta)-B(h,\delta)\sin^2\theta}\,\dif\theta.
\end{equation}

We use the standard complete elliptic integrals \cite[Section~19.2]{DLMF}:
\begin{equation}\label{eq:complete_elliptic_integrals}
K(k)=\int_0^{\pi/2}\frac{\dif\theta}{\sqrt{1-k^2\sin^2\theta}},
\qquad
E(k)=\int_0^{\pi/2}\sqrt{1-k^2\sin^2\theta}\,\dif\theta.
\end{equation}

\begin{rem}\label{rem:imaginary_modulus_branch}
The elliptic integrals are taken on the principal branch. If $k^2<0$, we
write $k=i\eta$, $\eta>0$, and use
\[
E(i\eta)
=
\sqrt{1+\eta^2}\,
E\!\left(\frac{\eta}{\sqrt{1+\eta^2}}\right),
\]
see \cite[Eq.~(19.7.2)]{DLMF}.
\end{rem}

\begin{thm}\label{thm:parameter_rectangle_integral}
For $0<c<1/2$, let $\mathcal E(h,c)$ be defined by
\eqref{eq:parameter_rectangle_integral}; for $c=0$, let
$\mathcal E(h,0)$ be defined by
\eqref{eq:lawson_parameter_rectangle_integral}. Set
\[
k^2=\frac{B(h,\delta)}{A(h,\delta)},
\]
where $A$ and $B$ are defined in \eqref{eq:AB}. Then
\begin{equation}\label{eq:parameter_integral_elliptic}
\mathcal E(h,c)
=
4\pi\sqrt{A(h,\delta)}\,E(k),
\end{equation}
where $E(k)$ is the complete elliptic integral in \eqref{eq:complete_elliptic_integrals}.
When $h>1$, one has $k^2<0$, and the principal branch is understood as in
Remark~\ref{rem:imaginary_modulus_branch}. In particular,
\begin{equation}\label{eq:clifford_parameter_integral}
\mathcal E(1,c)=2\pi^2
\qquad (0\leq c<1/2).
\end{equation}
\end{thm}

\begin{proof}
\eqref{eq:parameter_integral_elliptic} follows directly from
\eqref{eq:parameter_integral_theta} and \eqref{eq:complete_elliptic_integrals}. If $h=1$, then $A(1,\delta)=1$ and
$B(1,\delta)=0$, so $E(0)=\pi/2$ gives
\eqref{eq:clifford_parameter_integral}.
\end{proof}

\subsection{Willmore energy of the quotient surfaces}

We now pass from the parameter-rectangle integral to the Willmore energies
of the quotient surfaces by the automorphism groups. For $0<c<1/2$, the result is determined by the index of the period lattice
$\Lambda_X$ in the rectangular lattice $\Lambda_{\mathrm{par}}$.

\begin{thm}\label{thm:total_energy_lattice_index}
Let $h\geq0$ and $0<c<1/2$, and suppose that
$\operatorname{Hel}_c^h$ is compact.
Let $\Lambda_X=\operatorname{Per}(X_c^h)$ be the period lattice
defined in \eqref{eq:translation_period_set}, and let
$Q_c^h=\mathbb R^2/\Lambda_X$ be the quotient torus defined by
\eqref{eq:positive_c_automorphism_quotient}. With $q=p/r$ and
$h=j/\nu$ as in \eqref{eq:reduced_closing_data}, the Willmore energy is computed by
\begin{equation}\label{eq:translation_quotient_energy}
\mathcal W(Q_c^h)
=
\operatorname{lcm}(r,\nu)\,\mathcal E(h,c),
\end{equation}
where $\mathcal E(h,c)$ is written in
\eqref{eq:parameter_integral_elliptic}.
\end{thm}

\begin{proof}
Let
\[
\bar X_c^h:Q_c^h\longrightarrow\mathbb S^3
\]
be the immersion induced by $X_c^h$. Since $\bar X_c^h$ is minimal,
\eqref{eq:willmore_functional} and
\eqref{eq:quotient_willmore_energy} identify its Willmore energy with its
area. Recall from
\eqref{eq:translation_period_set} that
\[
\Lambda_{\mathrm{par}}
=
\langle(L,0),(0,2\pi)\rangle_{\mathbb Z}.
\]
Fix $s_0\in\mathbb R$ and let
\[
F_0=[s_0,s_0+L)\times[0,2\pi)
\]
be its standard fundamental domain. By Lemma~\ref{lem:period_lattice_index},
\[
N:=[\Lambda_{\mathrm{par}}:\Lambda_X]
=\operatorname{lcm}(r,\nu).
\]
Choose representatives
$\eta_1,\ldots,\eta_N\in\Lambda_{\mathrm{par}}$ for the quotient
$\Lambda_{\mathrm{par}}/\Lambda_X$. Then, up to sets of measure zero,
\[
F_X
=
\bigcup_{i=1}^N(F_0+\eta_i)
\]
is a fundamental domain for $\Lambda_X$.

The area density $\alpha(z(s))$, given by
\eqref{eq:alpha2_final}, is $\Lambda_{\mathrm{par}}$-invariant, and
\eqref{eq:parameter_rectangle_integral} gives
\[
\int_{F_0}\alpha(z(s))\,\dif s\,\dif t=\mathcal E(h,c).
\]
Therefore
\[
\mathcal W(Q_c^h)
=
\sum_{i=1}^N
\int_{F_0+\eta_i}\alpha(z(s))\,\dif s\,\dif t
=
N\,\mathcal E(h,c).
\]
Using $N=\operatorname{lcm}(r,\nu)$ gives
\[
\mathcal W(Q_c^h)
=
\operatorname{lcm}(r,\nu)\,\mathcal E(h,c),
\]
which proves \eqref{eq:translation_quotient_energy}.
\end{proof}

For $c=0$, the automorphism group gives a simpler energy
formula.

\begin{thm}
\label{thm:lawson_quotient_energy}
Let $h=j/\nu>0$, where $j$ and $\nu$ are coprime positive integers, and
let $Q_0^h$ be the quotient surface defined by
\eqref{eq:lawson_automorphism_quotient}. Then
\begin{equation}\label{eq:lawson_quotient_energy}
\mathcal W(Q_0^h)
=
\nu\,\mathcal E(h,0).
\end{equation}
\end{thm}

\begin{proof}
Let
\[
\bar X_0^h:Q_0^h\longrightarrow\mathbb S^3
\]
be the immersion induced by $X_0^h$. Since $\bar X_0^h$ is minimal,
\eqref{eq:willmore_functional} and
\eqref{eq:quotient_willmore_energy} identify its Willmore energy with its
area. Set
\[
a(s):=\sqrt{h^2\cos^2s+\sin^2s},
\qquad
a(s+\pi)=a(s).
\]
By \eqref{eq:lawson_parameter_rectangle_integral},
\[
\mathcal E(h,0)
=
2\pi\int_0^\pi a(s)\,\dif s.
\]

Suppose first that $j$ and $\nu$ are both odd. By
Proposition~\ref{prop:lawson_complete_symmetry_group}, the
automorphism group is the translation group associated with the lattice
generated by
\[
v_1=(\pi,\pi\nu),
\qquad
v_2=(0,2\pi\nu).
\]
The half-open parallelogram
\[
F=\{xv_1+yv_2:0\le x,y<1\}
\]
is therefore a fundamental domain for the automorphism group. Since
the area density depends only on the first coordinate,
\[
\begin{aligned}
\mathcal W(Q_0^h)
&=
\int_0^1\int_0^1 a(\pi x)\,2\pi^2\nu\,\dif y\,\dif x\\
&=
2\pi\nu\int_0^\pi a(s)\,\dif s
=
\nu\,\mathcal E(h,0).
\end{aligned}
\]

Suppose now that exactly one of $j$ and $\nu$ is even. A fundamental
domain for the translation subgroup $\mathcal T_0^h$ defined in
\eqref{eq:lawson_translation_subgroup} is determined by
\[
v_1=(2\pi,0),
\qquad
v_2=(0,2\pi\nu),
\]
and the integral of the area density over this domain is
\[
2\pi\nu\int_0^{2\pi}a(s)\,\dif s
=
4\pi\nu\int_0^\pi a(s)\,\dif s
=
2\nu\,\mathcal E(h,0).
\]
By \eqref{eq:lawson_group_mixed_parity},
$\mathcal T_0^h$ has index two in the automorphism group
$\mathcal G_0^h$. Since the area density is
$\mathcal G_0^h$-invariant, its integral over a fundamental domain for
$\mathcal G_0^h$ is half the preceding integral. Hence
\[
\mathcal W(Q_0^h)
=
\nu\,\mathcal E(h,0),
\]
which proves \eqref{eq:lawson_quotient_energy}.
\end{proof}

Theorems~\ref{thm:parameter_rectangle_integral},
\ref{thm:total_energy_lattice_index}, and
\ref{thm:lawson_quotient_energy} together prove
Theorem~\ref{mainthm:willmore}.

\subsection{Elliptic-integral estimates}

We collect the elliptic-integral estimates used in the monotonicity and
energy bounds below. The complete elliptic integrals $K$ and $E$ are defined
in \eqref{eq:complete_elliptic_integrals}. For $0<k<1$,
we use the standard derivative identities
\cite[Eqs.~(19.4.1)--(19.4.2)]{DLMF}
\begin{equation}\label{eq:elliptic_derivatives}
\begin{aligned}
E'(k)
&=\frac{E(k)-K(k)}{k},\\
K'(k)
&=\frac{E(k)}{k(1-k^2)}-\frac{K(k)}{k},\\
\frac{\dif E}{\dif(k^2)}
&=\frac{E(k)-K(k)}{2k^2}.
\end{aligned}
\end{equation}

\begin{lem}\label{lem:E_lower_bound}
For $0\le k\le1$,
\[
    E(k)^2\ge 2-k^2,
\]
with equality if and only if $k=1$.
\end{lem}

\begin{proof}
Let
\[
    G(k)=E(k)^2+k^2-2.
\]
Then $G(1)=0$. For $0<k<1$, using
\eqref{eq:elliptic_derivatives}, we obtain
\[
    G'(k)
    =
    2k-\frac{2E(k)\bigl(K(k)-E(k)\bigr)}{k}.
\]
Equivalently,
\[
    G'(k)
    =
    2k\left(
    1-
    E(k)\frac{K(k)-E(k)}{k^2}
    \right).
\]
By the Cauchy--Schwarz inequality,
\[
\begin{aligned}
E(k)\frac{K(k)-E(k)}{k^2}
&=
\left(\int_0^{\pi/2}\sqrt{1-k^2\sin^2\theta}\,\dif\theta\right)
\left(\int_0^{\pi/2}
\frac{\sin^2\theta}{\sqrt{1-k^2\sin^2\theta}}\,\dif\theta\right)\\
&\ge
\left(\int_0^{\pi/2}\sin\theta\,\dif\theta\right)^2
=
1.
\end{aligned}
\]
The inequality is strict for $0<k<1$: equality in
Cauchy--Schwarz would require
$\sqrt{1-k^2\sin^2\theta}$ to be proportional almost everywhere to
$\sin^2\theta/\sqrt{1-k^2\sin^2\theta}$, equivalently
$1-k^2\sin^2\theta=C\sin^2\theta$ for a constant $C$, which is impossible
on $[0,\pi/2]$. Hence $G'(k)<0$ on $(0,1)$.
Therefore, since $G(1)=0$, we have
\[
    G(k)>0
\]
for $0\le k<1$. This proves
\[
    E(k)^2\ge 2-k^2,
\]
with equality if and only if $k=1$.
\end{proof}

\begin{lem}\label{lem:elliptic_inequality}
For $0<k<1$,
\[
    (2-k^2)K(k)>2E(k).
\]
\end{lem}

\begin{proof}
Set $\Psi(k)=(2-k^2)K(k)-2E(k)$. This function extends
continuously to $k=0$ with $\Psi(0)=0$, and
\eqref{eq:elliptic_derivatives} gives
\[
    \Psi'(k)=k\left(\frac{E(k)}{1-k^2}-K(k)\right).
\]
Moreover
\[
    \frac{E(k)}{1-k^2}-K(k)
    =\frac{k^2}{1-k^2}
      \int_0^{\pi/2}\frac{\cos^2\theta}{\sqrt{1-k^2\sin^2\theta}}\,\dif\theta
    >0.
\]
Thus $\Psi'>0$ on $(0,1)$. Since $\Psi(0)=0$, it follows that
$\Psi(k)>0$ for every $0<k<1$.
\end{proof}

\begin{lem}\label{lem:boundary_function}
For $0<\delta<1$, set
\begin{equation}\label{eq:boundary_function_definition}
    \mathcal F(\delta)
    =\sqrt{\frac{1+\delta}{2}}\,
     E\left(\sqrt{\frac{2\delta}{1+\delta}}\right).
\end{equation}
Then $\mathcal F(\delta)>1$, and $\mathcal F$ is strictly decreasing in $\delta$.
\end{lem}

\begin{proof}
Let $k^2=2\delta/(1+\delta)$. Then
\[
    \delta=\frac{k^2}{2-k^2},
    \qquad
    \mathcal F(\delta)=\frac{E(k)}{\sqrt{2-k^2}}.
\]
Lemma~\ref{lem:E_lower_bound} gives $\mathcal F(\delta)>1$.  Differentiating,
\[
    \frac{\dif}{\dif k}\left(\frac{E(k)}{\sqrt{2-k^2}}\right)
    =\frac{2E(k)-(2-k^2)K(k)}{k(2-k^2)^{3/2}}<0
\]
by Lemma~\ref{lem:elliptic_inequality}. Since $\delta=k^2/(2-k^2)$ is
strictly increasing in $k$, the assertion follows.
\end{proof}

\subsection{Monotonicity and energy bounds}

\begin{thm}\label{thm:E_monotonic_h}
For each fixed $c\in[0,1/2)$, the function
$h\mapsto \mathcal E(h,c)$ is strictly increasing on $[0,1]$.
\end{thm}

\begin{proof}
Put $\delta=\sqrt{1-4c^2}$ and, for $0<h<1$, set
\begin{equation}\label{eq:Sk}
  S:=A(h,\delta),
\qquad
k=\sqrt{\frac{B(h,\delta)}{S}}, 
\end{equation}
where $A$ and $B$ are defined in \eqref{eq:AB}.
For $0<h<1$ and $0\le c<1/2$, one has $0<k<1$. By
\eqref{eq:parameter_integral_elliptic},
$\mathcal E(h,c)=4\pi\sqrt S\,E(k)$. Using
\eqref{eq:elliptic_derivatives} together with
\[
    S_h=h(1-\delta),
    \qquad
    (k^2)_h=-\frac{2h\delta}{S^2},
\]
we obtain
\[
\frac{\partial \mathcal E}{\partial h}
=\frac{2\pi h}{\sqrt S}
\left((1-\delta)E(k)+\frac{2\delta(K(k)-E(k))}{Sk^2}\right).
\]
The first term in parentheses is nonnegative, while the second is
strictly positive because $\delta>0$, $0<k<1$, and $K(k)>E(k)$.
Hence $\mathcal E$ is strictly increasing on $(0,1)$.

It remains to include the endpoints. Fix $0<h<1$, and choose
$h_-\in(0,h)$ and $h_+\in(h,1)$. Strict monotonicity on $(0,1)$
gives
\[
    \mathcal E(h_-,c)<\mathcal E(h,c)<\mathcal E(h_+,c).
\]
Continuity at $h=0$ and $h=1$ follows from the integral representation
\eqref{eq:parameter_integral_theta}. Together with monotonicity on $(0,1)$, taking limits $h_-\rightarrow 0+$ and $h_+\rightarrow1-$, we get
\[
    \mathcal E(0,c)<\mathcal E(h,c)<\mathcal E(1,c)=2\pi^2,
\]
and strict monotonicity holds on the closed interval $[0,1]$.
\end{proof}

\begin{thm}\label{thm:E_monotonic_c}
For fixed $0\leq h<1$, the function
$c\mapsto \mathcal E(h,c)$ is strictly increasing on $[0,1/2)$. For
$h=1$, the constant value $\mathcal E(1,c)\equiv 2\pi^2$ is given by
\eqref{eq:clifford_parameter_integral}.
\end{thm}

\begin{proof}
Set $\delta=\sqrt{1-4c^2}$. Since $\delta$ is strictly decreasing in $c$,
it suffices to prove that $\mathcal E$ is strictly decreasing in $\delta$ on
$0<\delta<1$; continuity then includes the endpoint $\delta=1$, equivalently
$c=0$.
Assume first that $0<h<1$, write $t=h^2$, and set $S$ and $k$ as in \eqref{eq:Sk}.
By \eqref{eq:parameter_integral_elliptic},
$\mathcal E=4\pi\sqrt S\,E(k)$. Using
\eqref{eq:elliptic_derivatives} together with
\[
    S_\delta=\frac{1-t}{2},
    \qquad
    (k^2)_\delta=\frac{1-t^2}{2S^2},
\]
we get
\[
\frac{\partial \mathcal E}{\partial\delta}
=
\frac{\pi}{\sqrt S}
\left((1-t)E(k)+\frac{(1-t^2)(E(k)-K(k))}{k^2S}\right).
\]
Using $1-t^2=(1-t)(1+t)$ and $1+t=S(2-k^2)$, the bracket above equals
\[
\frac{1-t}{k^2}
\bigl(2E(k)-(2-k^2)K(k)\bigr),
\]
which is negative by Lemma~\ref{lem:elliptic_inequality}. Thus
$\mathcal E$ is strictly decreasing in $\delta$ for $0<h<1$.

For $h=0$, \eqref{eq:parameter_integral_elliptic} and
\eqref{eq:boundary_function_definition} give
\[
    \frac{\mathcal E(0,c)}{4\pi}=\mathcal F(\delta).
\]
The conclusion follows from Lemma~\ref{lem:boundary_function},
together with \eqref{eq:sphere_parameter_integral}. The case $h=1$ is
\eqref{eq:clifford_parameter_integral}.
\end{proof}

\begin{coro}\label{cor:energy_bounds}
If $0\le h\le1$ and $0\le c<1/2$, then
\[
    4\pi\le \mathcal E(h,c)\le2\pi^2.
\]
The upper bound is attained exactly when $h=1$, and the lower bound only at
$(h,c)=(0,0)$.
\end{coro}

\begin{proof}
By Theorem~\ref{thm:E_monotonic_h},
\[
\mathcal E(0,c)\leq\mathcal E(h,c)\leq\mathcal E(1,c)=2\pi^2.
\]
By Theorem~\ref{thm:E_monotonic_c} and
\eqref{eq:sphere_parameter_integral},
$\mathcal E(0,c)>4\pi$ for $0<c<1/2$.
The equality cases follow immediately.
\end{proof}

\section{Lawson Associated Families and Finiteness under an Energy Bound}
\label{sec:associated_families}

This section uses the identification of the Lawson associated family of a
spherical catenoid with helicoidal minimal surfaces in
\cite[Theorem~5.6]{Castro2024}, derives the scaling law for the area integral
$\mathcal E$ defined in \eqref{eq:parameter_rectangle_integral} and
\eqref{eq:lawson_parameter_rectangle_integral}, and proves the finiteness
theorem under a prescribed Willmore energy bound.

\subsection{\texorpdfstring{Associated-family parameters and scaling of $\mathcal E$}{Associated-family parameters and scaling of E}}

We use the associated family in Lawson's sense
\cite[Section~13]{Lawson1970}. For a simply connected minimal immersion
$X:S\to\mathbb S^3$, after fixing the initial data, this gives a family
$X_\theta$ of minimal isometric immersions indexed by
$\theta\in\mathbb R/2\pi\mathbb Z$, with $X_0=X$ and
$X_{\pi/2}$ the conjugate immersion. For spherical catenoids, the interval
$0\leq\theta\leq\pi$ represents the associated family up to ambient
congruence, and Castro et al.\ identify these associated immersions with
helicoidal minimal immersions \cite[Theorem~5.6]{Castro2024}.

Fix $\beta\in(0,\pi/2)$, and let $\Cat_\beta$ denote the
spherical catenoid introduced in
\cite[Corollary~3.5 and Definition~3.6]{Castro2024}, corresponding to $\operatorname{Hel}_c^h$ with
$h=0$ and $c=\frac12\sin\beta$. By a harmless abuse of notation, we
also use $\Cat_\beta:\mathbb R^2\longrightarrow\mathbb S^3$ to denote
its global parametrization $X_c^h$ from
Definition~\ref{def:global_parametrizations}. Since its parameter domain
$\mathbb R^2$ is simply connected, let
$\{(\Cat_\beta)_\theta:\theta\in[0,\pi]\}$ denote the chosen
representatives of its Lawson associated family. For $0\leq\theta<\pi$, the relations
\cite[Theorem~5.6 and Eq.~(5.8)]{Castro2024} may be written as
\begin{equation}\label{eq:castro_family_relations}
    \sin\beta\,\sin\theta
    =
    \frac{2h}{1+h^2},
    \qquad
    \sin\beta\,\cos\theta
    =
    \frac{2c(1-h^2)}{1+h^2}.
\end{equation}
For $0<\theta<\pi$, the first relation in
\eqref{eq:castro_family_relations} gives two positive reciprocal roots for
$h$. Away from $\theta=\pi/2$, the sign of the second relation shows that
$c>0$ is obtained by choosing the root with $0<h<1$ on $(0,\pi/2)$ and
the root with $h>1$ on $(\pi/2,\pi)$. At $\theta=\pi/2$, the second
relation gives $c=0$, corresponding to the Lawson spherical helicoid. Thus, for $0<\theta<\pi$, $\theta\neq\pi/2$,
\cite[Theorem~5.6]{Castro2024} identifies $(\Cat_\beta)_\theta$, up to
ambient congruence, with the helicoidal minimal surface
$\operatorname{Hel}_{c(\theta)}^{h(\theta)}$. Throughout this section, we use the global parametrization
$X_{c(\theta)}^{h(\theta)}$ from
Definition~\ref{def:global_parametrizations} as the chosen representative
of this associated surface and take the quotient by
$\operatorname{Aut}(X_{c(\theta)}^{h(\theta)})$. Here
\begin{equation}\label{eq:lawson_h_theta}
h(\theta)
=
\begin{cases}
\displaystyle
\frac{1-\sqrt{1-\sin^2\beta\,\sin^2\theta}}
{\sin\beta\,\sin\theta},
&
0<\theta<\dfrac{\pi}{2},\\[10pt]
\displaystyle
\frac{1+\sqrt{1-\sin^2\beta\,\sin^2\theta}}
{\sin\beta\,\sin\theta},
&
\dfrac{\pi}{2}<\theta<\pi,
\end{cases}
\end{equation}
and
\begin{equation}\label{eq:lawson_c_theta}
    c(\theta)
    =
    \frac{\sin\beta\,\cos\theta}{2}
    \frac{1+h(\theta)^2}{1-h(\theta)^2}.
\end{equation}
\begin{rem}
The associated surfaces at $\theta=0$ and $\theta=\pi$ are congruent
spherical catenoids. For both endpoints we choose the convenient congruent
representatives
\begin{equation}\label{eq:associated_endpoint_convention}
h(0)=h(\pi)=0,
\qquad
c(0)=c(\pi)=\frac12\sin\beta.
\end{equation}
At $\theta=\pi$, the endpoint is represented by the congruent rotational
parametrization with $h=0$; the branch $h(\theta)>1$ is used only on
$(\pi/2,\pi)$. The formulas
\eqref{eq:lawson_h_theta}--\eqref{eq:lawson_c_theta} apply on
$(0,\pi)\setminus\{\pi/2\}$.

At $\theta=\pi/2$, one has $c=0$, and the two limiting helicoidal
parametrizations have reciprocal pitches
\begin{equation}\label{eq:lawson_conjugate_pitch}
    h_-=\tan\frac{\beta}{2},\qquad
    h_+=\cot\frac{\beta}{2}.
\end{equation}
They give two helicoidal parametrizations of the conjugate Lawson spherical
helicoid; see
\cite[Theorem~5.4, Theorem~5.6, and Remark~5.7]{Castro2024}. For
definiteness, we set
\begin{equation}\label{eq:associated_midpoint_convention}
    h(\pi/2):=h_-,
    \qquad
    c(\pi/2):=0.
\end{equation}
More generally, for $0<\theta<\pi$ with $\theta\neq\pi/2$, the two
roots of
\[
    (\sin\beta\,\sin\theta)h^2-2h+\sin\beta\,\sin\theta=0
\]
have product $1$, so
\begin{equation}\label{eq:h_theta_reciprocal}
    h(\pi-\theta)=\frac{1}{h(\theta)}.
\end{equation}
Since $\sin\beta\,\sin\theta<1$, neither solution attains the value $h=1$.
\end{rem}

By \cite[Theorem~5.6]{Castro2024} and the signs in
\eqref{eq:lawson_c_theta},
\[
h(\theta)>0,
\qquad
0<c(\theta)<\frac12
\]
on $(0,\pi)\setminus\{\pi/2\}$. We may therefore define
\[
\delta(\theta):=\sqrt{1-4c(\theta)^2}.
\]
Eliminating $\theta$ from \eqref{eq:lawson_c_theta} by \eqref{eq:castro_family_relations} gives
\[
4c(\theta)^2
=
\frac{
\sin^2\beta\bigl(1+h(\theta)^2\bigr)^2-4h(\theta)^2
}{
\bigl(1-h(\theta)^2\bigr)^2
}.
\]
It follows that
\begin{equation}\label{eq:delta_theta}
\delta(\theta)
=
\frac{1+h(\theta)^2}{|1-h(\theta)^2|}\cos\beta.
\end{equation}
This relation yields a scaling law for the area integral $\mathcal E$.

\begin{coro}\label{cor:energy_along_family}
For fixed $\beta\in(0,\pi/2)$ and $0<\theta<\pi$, $\theta\neq\pi/2$,
\begin{equation}\label{eq:associated_energy_ratio}
    \mathcal E\bigl(h(\pi-\theta),c(\pi-\theta)\bigr)
    =
    \frac{1}{h(\theta)}
    \mathcal E\bigl(h(\theta),c(\theta)\bigr).
\end{equation}
At $\theta=\pi/2$, the continuous extensions from the two parameter
intervals satisfy
\[
\mathcal E(h_+,0)
=
\frac{1}{h_-}\mathcal E(h_-,0),
\qquad
h_-=\tan\frac{\beta}{2},
\quad
h_+=\cot\frac{\beta}{2}=\frac{1}{h_-}.
\]
Moreover, on each of the intervals $(0,\pi/2)$ and $(\pi/2,\pi)$,
\begin{equation}\label{eq:associated_integral_scaling}
    \mathcal E\bigl(h(\theta),c(\theta)\bigr)
    =C(\beta)\sqrt{1+h(\theta)^2},
\end{equation}
where
\begin{equation}\label{eq:associated_energy_constant}
    C(\beta)=
    4\pi
    \sqrt{\frac{1+\cos\beta}{2}}\,
    E\!\left(\sqrt{\frac{2\cos\beta}{1+\cos\beta}}\right).
\end{equation}
Here $E$ is the complete elliptic integral of the second kind defined in
\eqref{eq:complete_elliptic_integrals}.
\end{coro}

\begin{proof}
Set $h=h(\theta)$, $c=c(\theta)$, and
$\delta=\delta(\theta)$. By
\eqref{eq:parameter_integral_elliptic},
\[
\mathcal E(h,c)=4\pi\sqrt{A(h,\delta)}\,E(k),
\qquad
k^2=\frac{B(h,\delta)}{A(h,\delta)}.
\]
If $0<\theta<\pi/2$, then $0<h<1$, and by definition \eqref{eq:AB} and
\eqref{eq:delta_theta}, we have
\[
A(h,\delta)=\frac{1+h^2}{2}(1+\cos\beta),
\qquad
B(h,\delta)=(1+h^2)\cos\beta.
\]
Thus \eqref{eq:associated_integral_scaling} holds with the constant
\eqref{eq:associated_energy_constant}.
If $\pi/2<\theta<\pi$, then $h>1$, and
\[
A(h,\delta)=\frac{1+h^2}{2}(1-\cos\beta),
\qquad
B(h,\delta)=-(1+h^2)\cos\beta.
\]
Here $k=i\eta$, where
$\eta^2=2\cos\beta/(1-\cos\beta)$. By
Remark~\ref{rem:imaginary_modulus_branch},
\[
\sqrt{A(h,\delta)}\,E(i\eta)
=
\sqrt{A(h,\delta)}\sqrt{1+\eta^2}\,
E\!\left(\frac{\eta}{\sqrt{1+\eta^2}}\right).
\]
Since
\[
\sqrt{A(h,\delta)}\sqrt{1+\eta^2}
=
\sqrt{\frac{1+h^2}{2}(1+\cos\beta)},
\qquad
\frac{\eta}{\sqrt{1+\eta^2}}
=
\sqrt{\frac{2\cos\beta}{1+\cos\beta}},
\]
the same scaling law follows on $\pi/2<\theta<\pi$. Combining it with \eqref{eq:h_theta_reciprocal} proves
\eqref{eq:associated_energy_ratio}.
The stated one-sided identity at $\theta=\pi/2$ follows from
\eqref{eq:lawson_conjugate_pitch} and the continuous extension of
\eqref{eq:associated_integral_scaling} along the two parameter intervals.
\end{proof}

For the two Lawson parametrizations with reciprocal pitches at
$\theta=\pi/2$, one is compact if and only if the other is. Indeed, if $h_-=j/\nu$ is
written in lowest terms, then $h_+=\nu/j$. Whenever they are compact,
denote their quotient surfaces by $Q_0^{h_-}$ and $Q_0^{h_+}$, as in
\eqref{eq:lawson_automorphism_quotient}. Then
Corollary~\ref{cor:energy_along_family} and
Theorem~\ref{thm:lawson_quotient_energy} give
\[
\begin{aligned}
\mathcal W(Q_0^{h_+})
&=j\,\mathcal E(h_+,0)\\
&=j\,\frac{1}{h_-}\mathcal E(h_-,0)\\
&=\nu\,\mathcal E(h_-,0)
 =\mathcal W(Q_0^{h_-}).
\end{aligned}
\]
Thus the compactness criterion and the Willmore energy of the quotient
surface by the automorphism group at $\theta=\pi/2$ do not depend
on which of the two Lawson parametrizations is chosen.

\subsection{Finiteness under a Willmore energy bound}

\begin{thm}\label{thm:finiteness}
Fix $\beta\in(0,\pi/2)$, and let $h(\theta)$ and $c(\theta)$ be the
helicoidal parameters defined by
\eqref{eq:lawson_h_theta}--\eqref{eq:lawson_c_theta} on
$(0,\pi)\setminus\{\pi/2\}$ and by
\eqref{eq:associated_endpoint_convention} and
\eqref{eq:associated_midpoint_convention} at
$\theta=0,\pi/2,\pi$. For every $W_0>0$, there are only finitely many
$\theta\in[0,\pi]$ for which the associated surface
$(\Cat_\beta)_\theta$ is compact and its quotient by the
automorphism group has Willmore energy at most $W_0$.
\end{thm}

\begin{proof}
By the preceding identification, it suffices to consider
$(0,\pi/2)\cup(\pi/2,\pi)$, since the three remaining parameter values
$0,\pi/2,\pi$ form a finite set. For a compact member with $0<c<1/2$, write
\[
q(h,c)=\frac{p}{r},
\qquad
h=\frac{j}{\nu}
\]
in lowest terms, as in \eqref{eq:reduced_closing_data}, and set
\[
    N:=\operatorname{lcm}(r,\nu).
\]
Let $\mathcal W$ denote the Willmore energy of its quotient surface by the
automorphism group. By
Theorem~\ref{thm:total_energy_lattice_index} and
\eqref{eq:associated_integral_scaling},
\[
    \mathcal W
    =N\mathcal E(h,c)
    =N\,C(\beta)\sqrt{1+h^2}.
\]
Moreover, the constant $C(\beta)$ defined in
\eqref{eq:associated_energy_constant} satisfies
\[
    C(\beta)
    =\mathcal E\left(0,\frac12\sin\beta\right)>4\pi
\]
by \eqref{eq:parameter_integral_elliptic} and
Corollary~\ref{cor:energy_bounds}. Hence $\mathcal W\le W_0$ implies
\[
    N<\frac{W_0}{4\pi},
\]
so only finitely many values of $N$ can occur.

For each fixed $N$, we have $\nu\mid N$ and
\[
    h^2
    \le
    \left(\frac{W_0}{N\,C(\beta)}\right)^2-1.
\]
Thus $\nu$ ranges over the finitely many divisors of $N$, and for each
such $\nu$, the numerator $j$ is bounded. Consequently, only finitely
many pitches $h=j/\nu$ are possible.

Finally, rewrite \eqref{eq:lawson_h_theta} as
\[
    h(\theta)=
    \begin{cases}
    \tan\bigl(\vartheta(\theta)/2\bigr),&0<\theta<\pi/2,\\[4pt]
    \cot\bigl(\vartheta(\theta)/2\bigr),&\pi/2<\theta<\pi,
    \end{cases}
    \qquad
    \vartheta(\theta):=\arcsin(\sin\beta\sin\theta).
\]
The function $\vartheta$ is strictly increasing on $(0,\pi/2)$ and strictly
decreasing on $(\pi/2,\pi)$. Since $\tan(\vartheta/2)$ is increasing and
$\cot(\vartheta/2)$ is decreasing, $h(\theta)$ is strictly increasing on each
parameter interval. Hence each admissible pitch $h(\theta)$ determines at most one
value of $\theta$ in each interval, and the set of admissible values of
$\theta$ is finite.
\end{proof}

This proves Theorem~\ref{mainthm:finiteness}.

\end{document}